\documentclass[reqno,11pt]{amsart}
\usepackage{euscript}
\usepackage{amssymb}
\usepackage{amscd}
\usepackage{amsmath}
\usepackage{graphics}
\usepackage{epsfig}
\usepackage{color}
\usepackage[all]{xy}
\usepackage{url}
\include{head}
\usepackage{pinlabel}
\usepackage[backref]{hyperref}

\numberwithin{equation}{section}
\newtheoremstyle{my}{1.5em}{0.5em}{\em}{}{\sc}{.}{0.5em}{}
\newtheoremstyle{your}{1.5em}{0.5em}{}{}{\sc}{.}{0.5em}{}

\theoremstyle{my}

\theoremstyle{my}
\newtheorem{thm}{Theorem}[section]
\newtheorem{Theorem}[thm]{Theorem}
\newtheorem*{Theorem*}{Theorem}
\newtheorem{Corollary}[thm]{Corollary}

\newtheorem*{corollary*}{Corollary}
\newtheorem{Lemma}[thm]{Lemma}

\newtheorem{Proposition}[thm]{Proposition}

\newtheorem*{conjecture*}{Conjecture}

\newtheorem*{question*}{Question}

\newtheorem*{definitions*}{Definitions}

\newtheorem*{rem*}{Remark}

\newtheorem*{remark*}{Remark}

\newtheorem*{remarks*}{Remarks}
\newtheorem*{example*}{Example}
\newtheorem*{examples*}{Examples}

\newtheorem*{convention*}{Convention}
\newtheorem*{conventions*}{Conventions}
\newtheorem*{Note*}{Note}
\newtheorem*{exercise*}{Exercise}
\newtheorem*{bibliographical-note*}{Bibliographical note}

\theoremstyle{your}
\newtheorem{Remark}[thm]{Remark}

\newcommand{\Acknowledgements}{{\em Acknowledgements.} }

\newcommand{\R}{\mathbb{R}}
\newcommand{\Z}{\mathbb{Z}}

\newcommand{\C}{\mathbb{C}}
\newcommand{\ev}{\operatorname{ev}}
\newcommand{\pr}{\operatorname{pr}}
\newcommand{\CZ}{\operatorname{CZ}}
\newcommand{\pa}{\partial}
\newcommand{\Ordo}{\mathcal{O}}

\newcommand{\bK}{\mathbb{K}}

\newcommand{\bR}{\mathbb{R}}
\newcommand{\bZ}{\mathbb{Z}}

\newcommand{\bC}{\mathbb{C}}

\renewcommand{\ker}{\mathrm{ker}}

\newcommand{\Hom}{\mathrm{Hom}}

\newcommand{\scrW}{\EuScript{W}}

\newcommand{\scrG}{\mathcal{G}}

\newcommand{\MM}{\mathcal{M}}
\newcommand{\FF}{\mathcal{F}}
\newcommand{\BB}{\mathcal{B}}

\newcommand{\NN}{\mathcal{N}}

\newcommand{\TT}{\mathcal{T}}

\newcommand{\DD}{\mathcal{D}}

\newcommand{\st}{\mathrm{st}}

\newcommand{\lk}{\operatorname{Hopf}}

\title{Nearby Lagrangian fibers and Whitney sphere links}
\author{Tobias Ekholm}
\address{Department of mathematics Uppsala University, Box 480, 751 06 Uppsala, Sweden \newline
Institut Mittag-Leffler, Aurav. 17, 182 60 Djursholm, Sweden}
\email{tobias.ekholm\@@math.uu.se}

\author{Ivan Smith}
\address{Centre for Mathematical Sciences, University of Cambridge, \newline Wilberforce Road, CB3 0WB, United Kingdom}
\email{is200\@@cam.ac.uk} 
\date{v1: September, 2016. v4: September, 2017.} 

\thanks{T.E. is partially supported by the Knut and Alice Wallenberg Foundation as a Wallenberg Scholar and by the Swedish Research Council.} 
\thanks{I.S. is partially supported by a Fellowship from the EPSRC}
\begin{document}
\thispagestyle{empty}

\begin{abstract} 
Let $n>3$, and let $L$ be a Lagrangian embedding of $\R^{n}$ into the cotangent bundle $T^{\ast}\R^{n}$ of $\R^n$ that agrees with the cotangent fiber $T_x^*\R^{n}$ over a point $x\neq 0$ outside a compact set. Assume that $L$ is disjoint from the cotangent fiber at the origin. The projection of $L$ to the base extends to a map of the $n$-sphere $S^n$ into $\R^{n}\setminus\{0\}$. We show that this map is homotopically trivial, answering a question of Y.\ Eliashberg. We give a number of generalizations of this result, including homotopical constraints on embedded Lagrangian disks in the complement of another Lagrangian submanifold, and on  two-component links of immersed Lagrangian spheres with one double point in $T^{\ast}\R^{n}$, under suitable dimension and Maslov index hypotheses. 
The proofs combine techniques from \cite{EkholmSmith, EkholmSmith2} with symplectic field theory.
\end{abstract}

\maketitle

\section{Introduction}\label{Sec:intro}
Recent work of Abouzaid and Kragh \cite{AbKr} has established the following striking rigidity result: if $Q$ is a closed manifold and $K\subset T^{\ast}Q$ is a closed exact Lagrangian submanifold, then the projection from $K$ to $Q$ is a simple homotopy equivalence. Here we consider related homotopy rigidity questions for Lagrangian disks in $T^*Q$ with prescribed behavior at infinity, and for Lagrangian immersions of spheres in Euclidean space with a single double point of high Maslov grading. Although our results are broadly inspired by the nearby Lagrangian submanifold conjecture, the methods of proof are very different: indeed, the homotopy equivalence of \cite{Abouzaid:htpy, AbKr}  is obtained from Whitehead's theorem, but here we focus on  situations where there is no underlying homological equivalence.

Our first result answers a question of Y.\ Eliashberg \cite{Yasha}. Let $x\ne 0$ be a point in $\R^{n}$ and let $L\subset T^{\ast}\R^{n}$ be a Lagrangian disk which agrees with the fiber $T^{\ast}_{x}\R^{n}$ outside a compact set. (By appropriate versions of the ``nearby Lagrangian submanifold conjecture", as established in \cite[Corollary 1.2]{Abouzaid:htpy} or \cite[Theorem 56]{EL},  any exact Lagrangian which agrees with the fiber outside a compact set must be a disk. See \cite[Corollary 3.12]{EkholmKraghSmith} for results on its parameterization, and Section \ref{Subsec:quasi-iso} for a related Floer-theoretic discussion.) Assume that $L$ is disjoint from the fiber $T_{0}^{\ast}\R^{n}$ at the origin. Then composing $L$ with the projection to the base we get a map $f_{L}\colon S^{n}\to\R^{n}\setminus \{0\}\simeq S^{n-1}$. For $n>3$ the homotopy group $\pi_n(S^{n-1})$ is isomorphic to $\Z/2\Z$.

\begin{Theorem}\label{Thm:yasha}
For $n>3$, the map $f_L$ represents the trivial element in $\pi_n(S^{n-1})$.
\end{Theorem}

If $0\in \R^n$ lies in the unbounded component of the complement of the projection of $L$ to $\R^n$, then $f_L$ is clearly nullhomotopic, but if $0$ lies in one of the bounded components of the complement, this result seems to have no elementary proof.  We prove Theorem \ref{Thm:yasha} by extending results of \cite{EkholmSmith,EkholmSmith2}. We compactify $L$ to a Lagrangian sphere  $\hat L$ which immerses in $T^*\R^{n}$ with one double point of Maslov grading $n$. A displacing Hamiltonian for $\hat L$ yields a 1-parameter family of deformations of the Cauchy-Riemann equations on the disk with boundary condition on $\hat L$. As in \cite{EkholmSmith2} we construct from the space of solutions a spin $(n+1)$-manifold $\mathcal{B}$ with $\partial{\mathcal{B}}=\hat L$. Using the space of holomorphic disks with boundary on $\hat L$, with one puncture at its double point, and a monotonicity argument, we extend the evaluation map on $\hat L=\partial{\mathcal{B}}$ to a map from $\mathcal{B}$ into the complement of $T^{\ast}_{0}\R^{n}$. In combination with the Pontryagin-Thom construction this gives the result.

The proof applies more generally. 
Let $S^{k}\subset\R^{k+1}\subset\R^{n}$ be the $k$-dimensional unit sphere for $0<k<n-1$ and let $\pr\colon T^{\ast} \R^{n}\to\R^{n}$ denote the bundle projection. Assume that $x\notin S^{k}$ and that $L$ is disjoint from $\pr^{-1}(S^k)$, the union of all cotangent fibers over points in $S^k$. Composing  the projection to $\R^{n}$ with the inclusion $\R^{n} \subset S^n$, where we now view $S^{n}$ as the one point compactification of the base, gives a map $f_L\colon S^{n}\to S^{n}\setminus S^k\simeq S^{n-k-1}$. Then the $m$-fold suspension gives a map 
\[  
\Sigma^{m}f_{L}\colon S^{n+m}\to S^{n-k-1+m}
\]
representing an element  $[\Sigma^{m}f_L]\in\pi_{n+m}(S^{n-k-1+m})$, which is independent of $m$ in the stable range $k<\frac12(n+m-3)$, the \emph{stable homotopy class} of $f_L$.

\begin{Corollary}\label{Cor:yasha}
The stable homotopy class of $f_{L}$ is trivial.
\end{Corollary}

For instance, cf. Remark \ref{Rmk:3dim},  a Lagrangian disk  $L \subset T^*(\R^3\setminus\{0\})$ which co-incides with a fiber near infinity projects to the zero-section with \emph{even} Hopf invariant; the question of whether the Hopf invariant vanishes integrally remains open. 

In the proofs of Theorem \ref{Thm:yasha} and Corollary \ref{Cor:yasha}, a key requirement is to disjoin the images of once-punctured holomorphic disks with boundary on $\hat L$ from the co-isotropic subset $\pr^{-1}(S^k)$. We ensure that disjointness via a monotonicity argument. When the subset to be avoided is a Lagrangian submanifold $C$, one can alternatively study the behavior of holomorphic curves in a neighborhood of $C$ by neck-stretching around $C$, as in symplectic field theory (SFT). Under additional assumptions on the geometry of $C$, index arguments ensure that the relevant holomorphic curves are disjoint from $C$ for sufficiently stretched almost complex structure. We next discuss two  results in this setting.

Let $\R^{2n}_{\st}$ denote\footnote{We will  write $\R^{2n}$ in place of $\R^{2n}_{\st}$ when context allows.} standard symplectic $2n$-space, namely $\R^{2n}$ equipped with the form $\omega_0=\sum_{j}dx_j\wedge dy_j$. Let $\phi\colon S^{n}\to\R^{2n}_{\st}$ be an immersed Lagrangian sphere with one double point of Maslov grading $n$. Let $C\subset \R^{2n}_{\st}$ be a Lagrangian submanifold with ideal Legendrian boundary $\Gamma$. We allow $C$ to be closed, corresponding to  $\Gamma=\varnothing$.

\begin{Theorem}\label{Thm:embLag}
Let $n\geq 4$. Suppose $C \subset \R^{2n}_{\st}$ is monotone with minimal Maslov number $\ge 3$ and admits a Riemannian metric for which the Morse index of any non-constant contractible geodesic loop is $\ge 3$. If $\phi(S^n)\subset \R^{2n}\setminus C$, then $\phi$ is null-homotopic in $\R^{2n}\setminus C$.
\end{Theorem}
   
In Appendices \ref{sec:algtop} and \ref{sec:withboundary} we compute $\pi_n(\R^{2n}\setminus C)$, which depends on the Stiefel-Whitney classes $w_i(TC)$ for $i=1,2$. 
Theorem \ref{Thm:embLag} is most interesting in the case that $w_2(TC)\in H^{2}(C;\Z/2\Z)$ vanishes, since then $\pi_{n}(\R^{2n}\setminus C)$ has a $\Z/2\Z$-subgroup corresponding to $\pi_n(\nu)$, where $\nu$ is the fiber $(n-1)$-sphere in the boundary of a tubular neighborhood of $C$. 

Manifolds satisfying the Morse index condition include (twisted) products of high dimensional spheres and manifolds without contractible geodesics.   (In the proof of Theorem \ref{Thm:embLag}, monotonicity arguments which prevent holomorphic discs escaping to infinity will show that the hypotheses on geodesics will only be required in a large compact subset of $C$, so there are no special requirements on the Riemannian metric near the ideal boundary.) 
In Section \ref{Sec:Linked} we give examples showing that the monotonicity and Maslov index hypotheses are necessary. 

We point out that Theorem \ref{Thm:embLag} strengthens Theorem \ref{Thm:yasha}. Noting that a disk has a metric without closed geodesics, the former result shows that we can replace the fiber $T^{\ast}_0\R^{n}$  by any Lagrangian disk $D$ which agrees with $[T,\infty)\times\Delta\subset [T,\infty)\times U^{\ast}\R^{n}\subset T^{\ast}\R^{n}$ for some $T>0$, where $U^{\ast}\R^{n}$ is the unit cotangent bundle and $[T,\infty)\times U^{\ast}\R^{n}$ denotes the complement of the radius $T$ disk cotangent bundle, such that for all sufficiently large $T'>0$ the intersection $L'=L\cap (T^{\ast}\R^{n}\setminus ([T',\infty)\times U^{\ast}\R^{n}))$ can be completed to a Lagrangian sphere with one double point of grading $n$ in the complement of $D$. This is possible for example if $\Delta$ lies in the restriction of the unit cotangent bundle $U^{\ast}\R^{n}|_{H_{x}}$ to the half space $H_{x}=\{u\in\R^{n}\colon u\cdot x\le 0\}$ of vectors with non-positive $x$-component. In this case we can extend $L'$ via a Lagrangian cobordism $L''$ with topology $[T',T'']\times S^{n-1}$, in the region $[T',T'']\times U^{\ast}\R^{n}$ over $\{u\in\R^{n}\colon u\cdot x\ge  x\cdot x\}$, which interpolates between $\partial L'=\partial_{-}L''$ and a Legendrian fiber sphere $\partial L''_{+}\subset\{T''\}\times U^{\ast}_{\lambda x}\R^{n}$ for some large $\lambda>0$. For $\lambda>0$ sufficiently large we are then far from $D$ and can cap the fiber sphere $\partial L''_{+}$ off with a standard half of a Whitney sphere in $T^{\ast}\R^{n}$ over $\{u\in\R^{n}\colon u\cdot x\ge \lambda x\cdot x\}$, see Sections \ref{sec:basic} and \ref{sec:comp}.  We point out that there is a rich class of non-standard Lagrangian disks $D$ (with $\Delta$ Legendrian knotted), see \cite{CNS} for examples when $n=2$, and \cite[Section 2.4]{E} for constructing analogous disks in higher dimensions.

Separately, we consider two component links of immersed spheres.  A \emph{Whitney sphere link} is an immersion $\iota\colon S^{n}\sqcup S^{n}\to\R^{2n}$ such that each component has exactly one double point and  the images of the components are disjoint. Such a link $\iota$ determines a Gauss map $\scrG_{\iota}\colon S^{n}\times S^{n}\to S^{2n-1}$,
\[ 
\scrG_{\iota}(x,y)=\frac{\iota(x)-\iota(y)}{|\iota(x)-\iota(y)|}\in S^{2n-1}.
\]
There are two homotopy classes of maps $S^{n}\times S^{n}\to S^{2n-1}$: 
\[
[S^n \times S^n, S^{2n-1}] \cong  \pi_{2n}(S^{2n-1}) = \bZ/2\Z,
\]
where $0\in\bZ/2\Z$ corresponds to the homotopy class of constant maps.
We define the \emph{Hopf linking number} $\lk(\iota)\in\bZ/2\Z$ as the homotopy class of $\scrG_{\iota}$.  A \emph{Lagrangian Whitney sphere link} is a Whitney sphere link which is a Lagrangian immersion.

\begin{Theorem} \label{Thm:whitney_link}
Let $n>4$ and let $\iota\colon S^n \sqcup S^n \to \bR^{2n}_{\st} $ be a Lagrangian Whitney sphere link. Assume that the Maslov grading of the double point of each component equals $n$. Then  $\lk(\iota)=0 \in \bZ/2\Z$.
\end{Theorem}

Theorem \ref{Thm:whitney_link} is derived from Theorem \ref{Thm:embLag} by taking $C$ to be a Lagrange surgery on one component of $\iota$. We point out that although the two components of $\iota$ appear symmetrically in the statement of the result, they play radically different roles in the proof. The restriction to $n>4$ is used to exclude certain degenerations of holomorphic curves for index reasons, cf. Section \ref{sec:stretch}. In Proposition \ref{Prop:2dim} we give an \emph{ad hoc} argument for Theorem \ref{Thm:whitney_link} when $n=2$; the cases $n=3,4$ remain open.

There are non-Lagrangian Whitney sphere  links in $\bR^4$ with non-zero Hopf linking number  \cite{FR, Massey-Rolfsen}.  
Applying results on Lagrangian caps from \cite{EM}, when $n\geq 3$ is odd one can construct Lagrangian Whitney sphere links (with double points of small Maslov index) with non-trivial linking number, see Lemma \ref{Lem:whitney_link}.    
 
The arguments used in the proofs of Theorems \ref{Thm:embLag} and  \ref{Thm:whitney_link} are closely related to arguments of Dimitroglou-Rizell and Evans \cite{DRE}, who prove homological non-linking theorems for monotone Lagrangian links whose components are diffeomorphic either to tori or to products $S^1 \times S^{n-1}$, and who further prove that the smooth isotopy type of a Lagrangian $S^1\times S^{n-1}$ in $\R^{2n}_{\st}$ is determined by its Lagrangian frame map when $n>4$. Our work, like theirs, relies on neck-stretching to localize holomorphic disks away from such monotone Lagrangian submanifolds, but we also appeal to the ``framed moduli spaces" machinery of \cite{EkholmSmith,EkholmSmith2}, and the Pontryagin-Thom construction to control homotopical rather than homological information.  As in \cite{DRE}, one can then infer results up to smooth isotopy. For example, any two Lagrangian Whitney sphere links for which the double points on each component are of Maslov index $n$ are smoothly ambient isotopic. Indeed, the nullhomotopy provided by Theorem \ref{Thm:whitney_link} implies that the links are formally isotopic, and the $h$-principle underlying  \cite[Theorem 1.3]{Ethesis} then yields a smooth ambient isotopy. Similarly, the homotopy result, Theorem \ref{Thm:yasha}, together with the $h$-principle in \cite{haefliger} shows that the Lagrangian disk $L$ is smoothly isotopic rel boundary to the fiber $T_{x}\R^{n}$ in the complement of the fiber $T_{0}\R^{n}$. For related smooth embedding results for (unparameterized) Lagrangian disks, see \cite[Corollary 3.10]{EkholmKraghSmith}.

\Acknowledgements  I.S. is grateful to Ailsa Keating and Oscar Randal-Williams for helpful conversations, to Yakov Eliashberg for asking the question, and to the Mittag-Leffler institute, where this project originated, for hospitality. T.E. thanks Thomas Kragh for useful discussions. The authors are grateful to the referee for suggesting clarifications to the exposition.

\section{Geometric stabilization for Lagrangian disks}
In this section we consider a stabilization procedure for Lagrangian disks in the setting of Corollary \ref{Cor:yasha} which corresponds to suspension of the induced map $f_{L}$.

\subsection{Conventions for Lagrangian disks with Legendrian boundary in $\R^{2n}$\label{Sec:leg_boundary}} 
In this section we introduce a specific convention for Lagrangian disks that is convenient for our study of associated homotopy classes.  
We consider a more general setting for Lagrangian disks in $\R^{2n}_{\st}$ that generalizes that considered above. Let $\Lambda$ be a Legendrian sphere in $\R^{2n-1}_{\st}=T^{\ast}\R^{n-1}\times \R$ with contact form $dz-y\cdot dx$, where $(x,y)$ are standard coordinates on $T^{\ast}\R^{n-1}$ and $z$ on $\R$. The symplectization of $\R^{2n-1}_{\st}$ is the symplectic manifold $\R\times\R^{2n-1}_{\st}$ with symplectic form $d(e^{t}(dz-y\cdot dx))$, where $t$ is a coordinate on the additional $\R$-factor. The symplectization then contains the Lagrangian cone $\R\times\Lambda$ on $\Lambda$. 

Consider embeddings of half of the symplectization into $\R^{2n}$ of the following form:
\begin{align*}
&\psi_{a}\colon [0,\infty)\times \R^{2n-1}_{\st} \ \to \ \ \R^{2n}_{\st}\cap \{x_{1}\ge 1\},\\
&\psi_{a}(t,x,y,z)=(e^{t},z,x,e^{t}y)+(a,0,0,0),\quad a>0. 
\end{align*}
     
Note that the image of $[0,\infty)\times\Lambda$ is a Lagrangian cylinder in $\{x_1\ge a\}$. (Here $\{x_1=a\}$ is a contact hypersurface and the Lagrangian cylinder lies in the corresponding cylindrical end.)

We will consider Lagrangian disks in $\R^{2n}$ with Legendrian boundary, which we define as Lagrangian embeddings of $\R^{n}$ that agree with a Lagrangian cylinder on a Legendrian sphere outside any disk of sufficiently large radius.

More precisely, let $\lambda\colon S^{n-1}\to\R^{2n-1}$ be a Legendrian embedding. Let $g\colon \R^{n}\to \R^{2n}$ be a Lagrangian embedding and let $(r,\xi)\in [0,\infty)\times S^{n-1}$ be polar coordinates on $\R^{n}$. If there exists $r_{0}>0$ and $a>0$ such that 
\[ 
g(r,\xi)=\psi_{a}\left(\log(r/r_{0}),\lambda(\xi)\right), \quad \mathrm{for} \, r\geq r_0,
\]
then $g$ is a \emph{Lagrangian disk with Legendrian boundary $\Lambda=\lambda(S^{n-1})$}.

\begin{Remark}
Lagrangian disks which agree with a cotangent fiber in $\R^{2n}\approx T^{\ast}\R^{n}$ can be isotoped to Lagrangian disks with Legendrian boundary as defined above. The Legendrian boundary of such a Lagrangian disk is the standard Legendrian unknot.
\end{Remark}

Let $L=g(\R^{n})$ be a Lagrangian disk with Legendrian boundary such that $g(\R^{n})\cap T^{\ast}_{S^{k}}\R^{n}=\varnothing$.  Let $h(r) = r/r_0 + a$. Then for, $r_1>0$ sufficiently large, $\pr\circ g|_{\{r\le r_{1}\}}$ gives a map from a disk $D$ into $\R^{n}\setminus S^{k}$ such that $D$ maps into $\{x_1\le h(r_{1})\}$, and $\partial D=\{r=r_1\}$ maps to $\{x_{1}=h(r_{1})\}$. Furthermore, there exists a constant $K>0$, depending on the Legendrian boundary of $L$ such that if $r_{1}$ is sufficiently large then $\pr(g(D))$ lies inside the radius $K\cdot h(r_{1})$ ball $B_{K\cdot h(r_{1})}$ around the origin in $\R^{n}$. In particular, we get a map
\[
\xymatrix{
{f_{L,r_{1}}\colon S^n=D/\partial D} \ar[rr]^-{\pr\circ g}  && (B_{K\cdot h(r_{1})}\setminus S^{k})/(\{x_1\ge h(r_{1})\}\cup\partial B_{K\cdot h(r_{1})} ) \,  \simeq \, S^{n-k-1}.
}
\]
It is clear that there exists $r_{1}>0$ such that the homotopy class of $f_{L,r}$ is independent of $r$ for all $r>r_{1}$. We write $[f_{L}]\in \pi_{n}(S^{n-k-1})$ for this homotopy class.  This  agrees with the corresponding homotopy class discussed in Section \ref{Sec:intro}.

\subsection{Half rotations and suspension\label{Subsec:half}}
Consider a Lagrangian disk $L$ with Legendrian boundary $\Lambda$ in $\R^{2n}_{\st}$ as above. We will construct from $L$ a Legendrian submanifold $\Gamma(L)\subset\R^{2n+1}_{\st}$, which can be thought of as a double of $L$ and which bounds a Lagrangian disk $C(L)$ in $\R^{2n+2}_{\st}$. 

We construct $\Gamma(L)$ by defining its front. Recall that if $\Sigma$ is a Legendrian submanifold of $J^{1}(\R^{n})=T^{\ast}\R^{n}\times \R$ then the front of $\Sigma$ is the 0-jet projection $f_{\Sigma}=\pr^{0}\colon \Sigma \to J^{0}(\R^{n})=\R^{n}\times\R$. Furthermore, if $\Sigma$ is in general position with respect to this projection then $f_{\Sigma}$ determines $\Sigma$ as follows. Let $\pr_{b}\colon\R^{n}\times\R\to\R^{n}$ denote the base projection. A generic point in the image of $\pr_{b}\circ f_{\Sigma}$ has a neighborhood $U$ such that $\pr_{b}^{-1}(U)\cap f_{\Sigma}(\Sigma)$ is the graph of a finite number of local functions $f_{j}\colon U\to\R$. The fiber coordinates $y=(y_1,\dots,y_n)$ are then determined over $U$ by the equations $y_j=\frac{\partial f_j}{\partial x_j}$, and the genericity condition ensures that these solutions continue over the singular locus (which has codimension one). Similarly, we define the front of an exact Lagrangian submanifold $L\subset\R^{2n}$ by picking a Legendrian lift into $\R^{2n}\times\R$ and taking the front projection into $\R^{n}\times\R$ of that lift. If $L$ is connected then its front is well-defined up to an over all $\R$-translation. 

We now turn to the actual construction, see Figure \ref{fig:double}. Consider the front $g_{L}$ of $L\cap\{x_1\le a\}$, for some $a>0$ sufficiently large that $L$ is a cone on its Legendrian boundary in the region $\{x_1\ge a-1\}$. Translating $L$ by $-a$ in the $x_{1}$-direction we get a front $g_{L}'$ lying over the half space $(-\infty,-1]\times \R^{n-1}\times\R$. If 
\[
f_{\Lambda}\colon S^{n-1}\to\R^{n-1}\times\R, \quad f_{\Lambda}=(f_{\Lambda}^{b},z_{\Lambda})
\]
 is the front for the Legendrian boundary $\Lambda$, then for $-2\le x_1\le -1$ the front $g'_{L}$ on $S^{n-1}\times[-2,-1]$ is  given by
\[ 
g'_{L}(\xi,r) = \left(r,f_{\Lambda}^{b}(\xi),\psi(r)z_{\Lambda}(\xi)\right),
\]
for a monotone increasing positive function $\psi$. Denote the Lagrangian disk in $\{x_{1}\le -1\}\subset\R^{2n}$ defined by the front $g'_{L}$ by $L^{-}$. 

The Lagrangian disk $L^{-}$ constitutes almost half of the Lagrangian projection of $\Gamma(L)$.  We next define the other half $L^{+}$: 
\[
L^{+}=T(L^{-}),
\]  
where $T\colon\R^{2n}\to\R^{2n}$ is the composition of reflections in the hyperplanes $\{x_1=0\}$ and $\{y_1=0\}$. The front of $L^{+}$  is then given by the function $g_{L}''$ on $S^{n-1}\times[1,2]$, where
\[ 
g_{L}''(\xi,r)=\left(r,f_{\Lambda}^{b}(\xi),\psi(-r)z_{\Lambda}(\xi)\right).
\]

With $L^{-}$ and $L^{+}$ defined we can now define the front of $\Gamma(L)$ in $\R^{n}\times\R$ as follows: 
\begin{enumerate}
\item over $\{x_1\le -1\}$ it agrees with the front of $L^{-}$, 
\item over $\{x_1\ge 1\}$ it agrees with the front of $L^{+}$, 
\item in the region $\{-2\le x_1\le 2\}$ the front is given by
\[ 
g_{L}(\xi,r)=\left(r,f_{\Lambda}^{b}(\xi),\phi(r)z_{\Lambda}(\xi)\right),
\]
where $\phi\colon[-2,2]\to\R$ is a Morse function with a maximum at $r=0$ and no other critical points, and such that $\phi(r)=\psi(r)$ for $-2<r<-1$, $\phi(r)=\psi(-r)$ for $1<r<2$.
\end{enumerate}

By construction, $\Gamma(L)$ is a Legendrian submanifold of $\R^{2n+1}$, whose Reeb chords all lie in the slice $\{x_1=0=y_1\}$, where they agree with the Reeb chords of $\Lambda$. (There are no Reeb chords of $\Gamma(L)$ outside the slice since the Lagrangian disk $L$, which is the projection along the Reeb direction, is embedded). If $c$ is a Reeb chord of $\Lambda$ then we denote by $\hat c$ the corresponding Reeb chord of $\Gamma(L)$. It follows from the ``front formula" for Maslov grading \cite[Lemma 3.4]{EES} that the gradings of the two Reeb chords are related by
\[ 
|\hat c|=|c|+1.
\]

We next construct a Lagrangian disk $C(L)\subset \R^{2n+2}_{\st}$ with Legendrian boundary $\Gamma(L)$, see Figure \ref{fig:double}. Again we construct it by defining its front. Consider the ``left" half of the front of $\Gamma(L)$, $g_{L}\colon D^{n}\to\R^{n}_{+}\times\R$, which lies over the half plane $\{x_1\le 0\}$.  Then $g_L$ takes the boundary $\partial D$ to the front of $\Lambda$ in $0\times\R^{n-1}\times\R$. 

Consider now $\R^{n+1}$, with coordinates $(x_1,x_2,\dots,x_n,x_{n+1})$, as 
an open book with binding $\R^{n-1}$ corresponding to the last $n-1$ coordinates and with page $\R^{n}_{+}$ with coordinates $(r\cos\theta,r\sin\theta,x_{3},\dots,x_{n+1})$, $r\ge 0$. Define the immersed Lagrangian disk $C'(L)$ to be that defined by the front which lies over the pages with positive $x_2$-coordinate (i.e., $0\le\theta\le \pi$) and which in each such page agrees with the front $g_{L}$. Then $C'(L)$ is an immersed Lagrangian, all of whose double points lie in the binding, these double points corresponding exactly to the Reeb chords of $\Lambda$. For small $\epsilon > 0$, the part of the front of $C'(L)$ which lies over $\{x_2\ge\epsilon\}$ defines an embedded Lagrangian disk, which near the boundary looks like a cone on the Legendrian sphere $\Gamma(L)$. After small deformation of the front we may then add a half-infinite cone on the front of $\Gamma(L)$ to obtain a Lagrangian disk with Legendrian boundary in the sense defined before. This is our desired Lagrangian disk $C(L)$.   

\begin{figure}[ht]
\centering
\includegraphics[width=.6\linewidth]{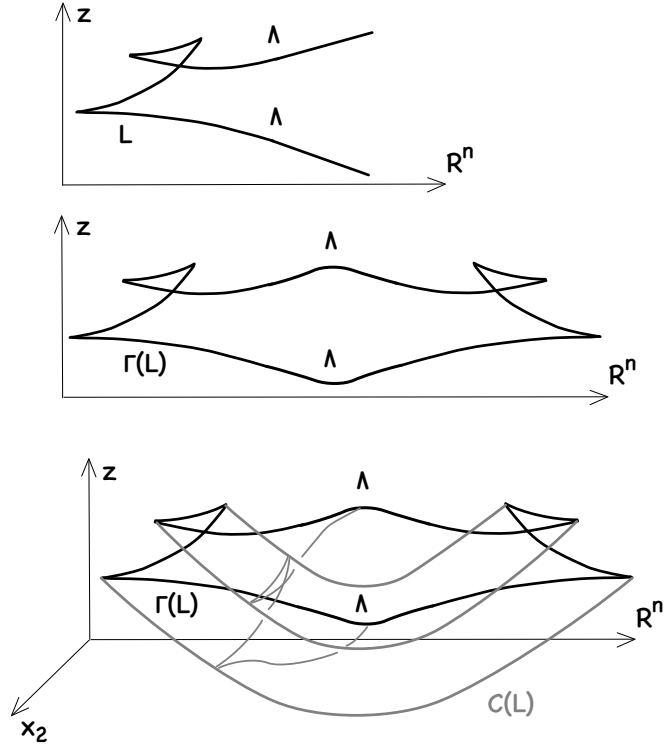}
	\caption{The fronts of: $L$ (top), $\Gamma(L)$ (middle) and $C(L)$ (bottom).}
	\label{fig:double}
\end{figure}

\begin{Lemma}
Let $L$ be a Lagrangian disk with Legendrian boundary $\Lambda$ in $\R^{2n}$.
Assume that $L$ is disjoint from $T^{\ast}\R^{n}|_{S^{k}}$. Then one can construct the Lagrangian disk $C(L) \subset \R^{2n+2}$, with boundary $\Gamma(L)$, so that $C(L)$ is disjoint from $T^{\ast}\R^{n+1}|_{S^{k}}$.
\end{Lemma}

\begin{proof}
Take $S^{k} \subset \R^{n+1}$ to lie in an $\R^{n}_{+}$-page of the open book decomposition of $\R^{n+1}$ used in the construction of $C(L)$.  It then follows immediately from the construction that $C(L)$ is disjoint from $T^{\ast}\R^{n+1}|_{S^{k}}$. 
\end{proof}

Recall that in this situation, we associated maps $f_L\colon S^{n}\to S^{n-k-1}$ and $f_{C(L)}\colon S^{n+1}\to S^{n-k}$ to $L$ and to $C(L)$, respectively.

\begin{Lemma} \label{Lem:suspension}
The map $f_{C(L)}\colon S^{n+1}\to S^{n-k}$ is homotopic to the suspension of the map $f_{L}\colon S^{n}\to S^{n-k-1}$ associated to $L$,  i.e. $f_{C(L)} \simeq \Sigma f_L$.
\end{Lemma}
 
\begin{proof}
To see this we use the Pontryagin-Thom construction. The preimage of a point in $S^{n-k-1}$ under $f_{L}$ corresponds to the intersection of $L$ and $T^{\ast}_{\ell}\R^{n}$, where $\ell$ is a half-infinite line perpendicular to $S^{k}\subset\R^{n}$. For generic $\ell$ this intersection is an orientable $(k+1)$-manifold, framed by a basis of the tangent space to the space of half-lines perpendicular to $S^{k}$ at the base point of $\ell$.

Similarly, the preimage of a point under $f_{C(L)}$ is given by the intersection of $C(L)$ and a ray $\ell$, which we can take to be the same ray as before. If we assume the ray $\ell$ lies in a page,  then it follows that the preimage is the same as that for $f_{L}$ but the normal framing is extended by the constant vector field $\partial_\theta$ (where $\theta$ is a co-ordinate in the parameter space of the pages of the open book). Such a constant stabilisation of the normal framing is exactly what happens with preimages of submanifolds of a sphere under suspension of maps. The lemma follows.   
\end{proof} 

\section{Floer-holomorphic disks on immersed Lagrangian spheres}\label{Sec:modulispaces}
A main ingredient in the proofs of all our results is the moduli space of Floer holomorphic disks with boundary on an immersed Lagrangian sphere in $\R^{2n}_{\st}$. This space was studied in detail in \cite{EkholmSmith2}. In this section we recollect and extend these results.  We assume $n>3$ unless explicitly stated otherwise.

\subsection{Lagrangian sphere immersions\label{sec:basic}} Let $\phi\colon S^{n}\to\R^{2n}_{\st}$ be a Lagrangian immersion. The tangential Gauss map takes any point $p\in S^{n}$ to the Lagrangian tangent plane $d\phi(T_{p}S^{n})$ and thus defines a map $G_{\phi}\colon S^{n}\to U(n)/O(n) \to U/O$, where $U(n)/O(n)$ is the Lagrangian Grassmannian of Lagrangian $n$-planes in $\R^{2n}_{\st}$, and the map to $U/O$ is given by stabilization. After small perturbation $\phi$ is self-transverse and hence has only transverse double points and no other self intersections. If $a$ is a transverse double point of $\phi$ then there is an integer Maslov grading $|a|$ associated to $a$, see \cite{EkholmSmith2}.

The standard example of a Lagrangian sphere immersion is the \emph{Whitney sphere}, defined as follows. Equip $\R^{2n}_{\st}$ with the standard complex structure $i$ and view it as complex  $n$-space $\C^{n}$ with coordinates $x+iy\in \R^n+i\R^{n}$. Consider $S^{n}$ as the unit sphere in $\R^{n+1}$: 
\[
S^n = \left\{(x,y_1) \in \R^n \times\R \ :  \ |x|^2 + y^2_1 = 1\right\}.  
\]
The Whitney immersion is the map $w\colon S^{n}\to\C^{n}$ given by
\begin{equation}\label{eq:whitneysphere} 
w(x,y) =  x(1+iy_1).
\end{equation}
Then $w$ has one self-transverse double point $c$ with preimages at $(0,\pm 1)$ and no other singularities. The Maslov grading of the double point is $|c|=n$.  

We will assume throughout this section that $\phi\colon S^{n}\to\R^{2n}_{\st}$ is a self transverse real analytic Lagrangian immersion with exactly one double point $a$ of Maslov grading $|a|=n$. We say that $\phi$ is \emph{tangentially standard} if its stable tangential Gauss map $G_{\phi}$ is homotopic to the stable tangential Gauss map $G_{w}$ of the Whitney immersion.

Fix a primitive $\theta$ of the symplectic form, $d\theta = \omega_{\st}$. 
The immersed sphere $\phi\colon S^n \to \R^{2n}_{\st}$ has a Legendrian lift $\phi \times z\colon S^n \to \R^{2n}_{\st} \times \R$, where the contact form on $\R^{2n}_{\st}\times \R$ is $dz-\theta$, defined with respect to a choice $z\colon S^n \to \R$ of function satisfying $\phi^*\theta = dz$. In the case under consideration $\phi$ has a single double point and $\phi\times z$ is an embedding.  We can then distinguish between small disjoint disk neighborhoods $V^{\pm}$ of the two preimages of the double point of $\phi(S^n)$, by declaring that $V^+$ lives in the upper sheet and $V^-$ in the lower sheet of the image of $\phi \times z$.

\subsection{Moduli spaces from displacing Hamiltonians} Fix an almost complex structure $J$ on $\R^{2n}_{\st}$ which is standard in a sufficiently small neighborhood of $\phi(S^n)$. Let $H=H_{t}\colon \R^{2n}\to \R$ be a time dependent Hamiltonian function with associated Hamiltonian vector field $X_{H_t}$ and time-one flow $\psi_H^1$.  We suppose that $\psi_H^1(\phi(S^n)) \cap \phi(S^n) = \varnothing$.  As in \cite[Section 3.1, Equation (3.3)]{EkholmSmith} we fix a 1-parameter family of 1-forms $\gamma_{r} \in \Omega^1(D)$ on the 2-dimensional closed disk $D$, with $r\in[0,\infty)$, such that $\gamma_0 \equiv 0$, and such that with respect to a fixed conformal isomorphism $D \setminus \{\pm1\} \to \R\times [0,1]$, and with co-ordinates $(s,t) \in \R\times [0,1]$, $\gamma_r$ has compact support and for $r\gg0$ agrees with $dt$ on $[-r,r]\times [0,1]$.  

Consider the Floer equation for maps $u\colon (D,\partial D)\to (\R^{2n}_{\st},\phi(S^{n}))$, such that $u|_{\pa D}$ admits a continuous lift to $S^{n}$,
\begin{equation}\label{eq:Floer}
(du + \gamma_{r}\otimes X_{H})^{0,1}=0, 
\end{equation} 
where the complex anti-linear part is taken with respect to  $J$ and the standard complex structure on $D$, and where $X_{H}=X_{H}(u(z),t(z))$ for the function $t\colon D\to[0,1]$ defined by the second co-ordinate in our fixed conformal equivalence $D-\{\pm 1\} \to\R\times[0,1]$. For fixed $r\in[0,\infty)$, we write $\FF^{r}$ for the space of solutions of \eqref{eq:Floer} and we write
\[ 
\FF=\bigcup_{r\in[0,\infty)}\FF^{r}
\]
for the corresponding parametrized moduli space.

\begin{Lemma}
	The space $\FF^{r}$ has formal dimension $n$; $\FF$ has formal dimension $n+1$. For generic data the formal dimension equals the actual dimension, and $\FF$ has boundary
	\[ 
	\partial\FF=\FF^{0}\approx_{C^{1}} S^{n}.
	\]
\end{Lemma}

\begin{proof}
	The dimension formula is a consequence of the index $n+\mu$ of the Riemann-Hilbert problem in $\C^{n}$ for a Lagrangian boundary condition of Maslov index $\mu$, and the fact that the Maslov index of $S^n$ is zero. The transversality result is  \cite[Lemma 3.4]{EkholmSmith2}. When $r=0$ there is no Hamiltonian term, the Floer equation reduces to the unperturbed Cauchy-Riemann equation, and $\FF^{0}$ consists of constant solutions by exactness.
\end{proof}

In order to describe the compactification of $\FF$, consider spaces $\FF_{j}^{r}$ of solutions $u$ of the Floer equation \eqref{eq:Floer} with $j$ negative boundary punctures, where the disk is asymptotic to the unique double point of $\phi(S^{n})$. More precisely, the source of such a map is the disk $D$ with $j$ boundary punctures $\zeta_1,\dots,\zeta_j$. Consider a punctured arc $I_j$ in $\pa D$ centered around $\zeta_j$ and note that it is subdivided into two components $I_{j}^{-}$ and $I_{j}^{+}$ in the negative and positive direction along the boundary from $\zeta_j$. We require that for any sufficiently small $I_j$, $u(I_j^{-})$ lies in the upper branch $V^+$ of the double point and $u(I_{j}^{+})$ in the lower branch $V^-$. The following results were proved in \cite[Lemma 3.4]{EkholmSmith2}.
\begin{Lemma}\label{Lem:vdim_floer}
	The formal dimensions of $\FF^{r}_{j}$ and $\FF_{j}$ are 
	\[ 
	n-j(n-1) \quad \text{ and }\quad n-j(n-1)+1,
	\]
	respectively. For generic data these spaces are all transversely cut out. In particular, $\FF_{1}$ is a closed one-dimensional manifold and $\FF_{j}$ is empty for $j>1$.
\end{Lemma}

Similarly, we consider the moduli space $\MM$ of unperturbed holomorphic disks with boundary on $\phi(S^{n})$  and one positive puncture at the double point. The following result is proved in \cite[Lemma 3.5]{EkholmSmith2}.
\begin{Lemma}
	The formal dimension of $\MM$ is $n-1$ 
	and for generic data the moduli space is transversely cut out. Furthermore, $\MM$ is a closed $C^{1}$-manifold.
\end{Lemma} 

The moduli space $\MM$ is a space of holomorphic maps modulo automorphisms, and the construction of the smooth structure above relies on a gauge-fixing procedure described in detail in \cite[Section A.2]{EkholmSmith2}; the gauge-fixing involves fixing parametrizations of maps in $\MM$ so that $\pm i$ map to small spheres centred on the preimages of the double point.  
We will not distinguish notationally between $\MM$ and the corresponding space of maps after gauge-fixing; in particular, in the sequel we will make use of an evaluation map $\MM \times D \to \C^n$, where $D$ denotes the source-disk for maps in the gauge-fixed moduli space corresponding to $\MM$.

\subsection{The Gromov-Floer compactification}  The space $\FF$ is non-compact; it has one boundary component $\FF^0$ which is diffeomorphic to the sphere $S^n$ parametrizing constant solutions. It can be compactified to a $C^1$-smooth compact manifold with boundary by adding broken solutions. 

\begin{Lemma}
	For generic $J$ and $H$, the product
	\[ 
	\NN=\FF_{1}\times \MM
	\]
	is a $C^1$-smooth $n$-manifold, canonically diffeomorphic to the Gromov-Floer boundary of $\FF$. 
\end{Lemma}

\begin{proof} 
This is proved in \cite[Lemma 3.6]{EkholmSmith2}. Note that	other broken configurations are ruled out because, for generic data, $\FF_{j}=\varnothing$ for $j>1$ (this relies on our standing assumption $n>3$).
\end{proof}

 \cite[Theorem 5.21]{EkholmSmith2} constructs a Floer gluing map on $\NN\times [\rho_0,\infty)$ which, when composed with a map reparameterizing the domain, gives a smooth embedding   
\begin{equation} \label{eqn:gluing}
\Psi\colon \NN\times [\rho_0,\infty) \to \FF
\end{equation}
whose image parameterizes a neighborhood of the Gromov-Floer boundary of $\FF$.  (The effect of the re-parameterization, discussed further below, is that the holomorphic disk part of the glued map, coming from $\MM$, lies in a small half-disk around the puncture of the domain of the map from $\FF_1$.) 
It follows that
\[ 
\overline{\FF}=\FF\setminus \Psi(\NN\times (\rho_0,\infty))
\] 
is a smooth compact submanifold of $\FF$ with boundary 
\[  
\partial\overline{\FF}  \approx_{C^{1}}  
S^n \cup \NN,
\]
where the diffeomorphism on the first component is given by inclusion of constant maps (with a lift of the map to the domain of the immersion $\phi$), and the diffeomorphism on $\NN$ is given by the reparameterized gluing map $\Psi$.

Recall that $\FF_{1}$ is a closed $1$-manifold. As in \cite[Section 4]{EkholmSmith2}
consider next the abstract filling
\[ 
\TT= \DD\times \MM
\]
of $\NN$, where $\DD$ is a collection of disks filling $\FF_{1}$.  We let $\BB = \overline{\FF} \cup_{\NN} \TT$, which is a $C^1$-smooth compact $(n+1)$-dimensional manifold with boundary, whose unique boundary component is canonically diffeomorphic to $S^n$.

\begin{Proposition}\label{Prop:parallel} 
If $\phi$ is tangentially standard then $\BB$ is parallelizable.  
\end{Proposition}

\begin{proof} Since $\BB$ is a manifold with boundary, it suffices to prove that it is stably parallelizable. This is \cite[Theorem 1.1]{EkholmSmith2}.  \end{proof}

With no assumption on the stable Gauss map of $\phi$ we have the following weaker result.  

\begin{Proposition}\label{Prop:spin}
	If $n>3$, the manifold $\BB$ is spin.	
\end{Proposition}

\begin{proof}
The tangent bundle $T\overline{\FF}$ of $\overline{\FF}$ is the restriction of the index bundle for the linearized Cauchy-Riemann equation. This index bundle is pulled back from the space of linearized boundary conditions over the free loop space of the sphere $S^{n}$. The loop space is $(n-2)$-connected,  hence $T\overline{\FF}$ is trivial over the $2$-skeleton since $n>3$, which means $\overline{\FF}$ is spin. 

The manifold $\BB$ is constructed from $\overline{\FF}$ by gluing $\DD\times\MM$ along $\FF_1\times\MM$. Consider a circle component $S\subset\FF_1$ and the corresponding filling $D\times \MM$, and let $m \in \MM$. The spin structure on $\FF$ admits an extension over $D\times\MM$ provided that the induced spin structure on $S\times\{m\}$ is the bounding spin structure (i.e.~ the spin nullcobordant spin structure, which in particular extends over a disk). Now, either $S\times \{m\}$ bounds in $\overline{\FF}$, in which case the induced spin structure bounds, or it does not. In the latter case, we alter the spin structure on $\overline{\FF}$ by the element in $H^{1}(\overline{\FF};\Z/2\Z)$ corresponding to $S\times\{m\}$, and then it does bound. The result follows.  	
\end{proof}

\begin{Remark} \label{Rem:spin_bound}
Arguing as in the proof of Proposition \ref{Prop:spin} we find that if $\gamma \subset \BB$ is a loop, we can choose a spin structure on $\BB$ that induces the nullcobordant framing on $\gamma$. 
Similarly, in the setting of Proposition \ref{Prop:parallel}, we can choose a stable trivialization so that the induced framing on $\gamma$ is nullcobordant. Indeed, the framing of $\BB$ is obtained from choices of stable framings of index bundles over configuration spaces containing $\FF$, $\FF_1$ and $\MM$, and in the construction one sees that one can choose the stable framing on the $1$-skeleton arbitrarily subject to the fact that it bounds on $1$-cells lying in $\FF_1$. Compare to \cite[Section 3.7]{EkholmSmith} and  \cite[Lemma 4.9]{EkholmSmith2}.
\end{Remark}

\subsection{Extending the evaluation map} From the  Lagrangian immersion $\phi\colon S^n \rightarrow \R^{2n}_{\st}$ we have constructed a bounding manifold $\BB$, $\partial \BB = S^n$. The subspace $\overline{\FF}\subset \BB$ is a space of Floer holomorphic disks with boundary that lift via $\phi$ to $S^n$, and hence there is a canonical evaluation map $\ev_1\colon \overline{\FF} \rightarrow S^n$.  This evaluation map cannot extend to $\BB$ as a map to $S^n$, because the map has degree $1$ on the boundary $\partial \BB$. However, it can be extended over the filling $\TT$ as a map to $\R^{2n}$. We  prove our results on homotopy classes using a particular such extension coming from the holomorphic disk components of broken solutions.

Consider first the 1-manifold $\FF_{1}$. Let $\pa D$ denote the boundary of the source disk $D$ of the maps in $\FF_{1}$. There is a smooth map
\[ 
\zeta\colon \FF_{1}\to \pa D
\]
taking a solution $u\in\FF_{1}$ to the coordinate of its negative boundary puncture. After a small rotation of the domain we may assume that $\zeta$ is transverse to $1\in\pa D$. Fix a small closed interval $I\subset \pa D$ centered at $1$, and let
\[ 
\zeta^{-1}(I)=I_1\cup\dots\cup I_m,
\] 
where $I_j$ are the components of the preimage and where $I$ is sufficiently small that the restriction of $\zeta$ to each $I_j$ is a diffeomorphism onto $I$. We subdivide $\NN=\FF_{1}\times\MM$ into $m+1$ pieces:
\[ 
\NN = \bigcup_{j=1}^{m} \left( I_j\times \MM\right) \ \cup \ \left ( U\times \MM \right),
\] 
where $U=\FF_{1}\setminus \bigcup_{j=1}^{m}I_j$. Recall $\FF_1 = \partial \DD$ for an abstract collection of 2-disks $\DD$. Let $\DD_{j}\subset \DD$ denote a small half disk neighborhood of the center of $I_j$ that intersects the boundary in $I_j$.  Recall the neighborhoods $V^{\pm}$ of the preimages of the double point introduced in Section \ref{sec:basic}.

\begin{Lemma}\label{l:safeeval}
	Let $\sigma_j \in \{\pm\}$ be the sign of the derivative $d\zeta$ in $I_j$ and let $\pa_{+} I_j$ and $\pa_{-}I_j$ denote the positive and negative endpoints of $I_j$. For $I$ small enough
	\[ 
	\ev_1(\pa_{-}I_j)\subset V^{\sigma_j}\quad \text{ and }\quad
	\ev_1(\pa_{+}I_j)\subset V^{-\sigma_j}. 
	\]  
\end{Lemma}

\begin{proof}
	This holds since, with respect to the orientation on $\pa D$ induced by the complex orientation on $D$, points in $\pa D$ just before the puncture map to $V^{+}$, whilst points right after the puncture map to $V^{-}$.
\end{proof}

We extend the evaluation map $\ev_1$ from $\overline{\FF}$ to $\BB$ in two stages.  The conditions of the following lemma will arise naturally in the setting of Theorem \ref{Thm:yasha}, but can always be achieved by appropriate scaling and translation. Let $(x,y)=(x_1,y_1,\dots,x_n,y_n)$ be standard coordinates on $\R^{2n}$. Let $L\subset \R^{2n}$ be a Lagrangian disk with Legendrian boundary $\Lambda$, and suppose that $L$ agrees with the cone on $\Lambda$ in the half space $\{x_1\le -\epsilon\}$. (Note that here we are conical in a region lying over the negative, rather than the positive, $x_1$-axis; this differs from the convention at the start of Section \ref{Sec:leg_boundary}, but fits with the construction of the Legendrian double $\Gamma(L)$ in Section \ref{Subsec:half}, which involves translation of the front.) Assume that $\Lambda$ has a single Reeb chord, which has grading $(n-1)$. We construct the Legendrian double $\Gamma(L)$ of $L$, in the sense of Section \ref{Subsec:half}, in such a way that (i) $L^{+}$ lies in $\{x_1\ge-\tau\}$, and agrees with the front of $L$ in that region; (ii) its reflection $L^{-}$ lies in $\{x_1\le -\tau-2\epsilon\}$; and (iii) the cylindrical piece which joins the two  lies in $\{-2\epsilon-\tau \le x_1\le-\tau\}$. Let $\phi\colon S^{n}\to\R^{2n}$ be the Lagrangian immersion obtained by projecting $\Gamma(L)$ along the $z$-axis. Then $\phi$ has one transverse double point, corresponding to the unique Reeb chord of $\Lambda$; the double point has grading $n$ and lies in the slice $\{x_1=-\tau-\epsilon\}$.   

As mentioned briefly above,  the map 
\[  
\Psi\colon \FF_{1}\times\MM\times[\rho_0,\rho_1]\to\overline{\FF}
\]
that gives the collar neighborhood on the boundary is constructed by first pre-gluing the disks in $\FF_1$ and $\MM$, then applying Floer-Picard iteration to obtain an actual solution, and finally composing the result with a reparameterization of the domain so that the holomorphic disk part lies in a small half-disk of size $\Ordo(e^{-\rho})$ close to the boundary puncture of the disk in $\FF_1$. In this construction, we use an explicit gauge-fixing procedure for the holomorphic maps in $\MM$. The (pre-)gluing map in fact takes as input an element in this gauge-fixed space of smooth maps, see \cite[Section A.9]{EkholmSmith2}. 

For $(u,v,\rho)\in \FF_1\times\MM\times[\rho_0,\rho_1]$, let $\Psi'(u,v,\rho)$ be the ``naive" version of  $\Psi(u,v,\rho)$ in which we do not apply Floer-Picard iteration: $\Psi'(u,v,\rho)$ is constructed only by pregluing and reparameterization. The distance between the starting point for Floer-Picard iteration and the actual solution resulting from iteration is estimated in \cite[Equation (5.10)]{EkholmSmith2}. In the current setting this implies that, by taking $\rho_0$ sufficiently large, we can make $\Psi(u,v,\rho)$ and $\Psi'(u,v,\rho)$ arbitrarily $C^{1}$-close for all $(u,v,\rho)\in\FF_1\times\MM\times[\rho_0,\rho_1]$.

\begin{Lemma} \label{Lem:extension_first_part}
For $\phi\colon S^{n}\to\R^{2n}$ a Lagrangian immersion with one double point of Maslov grading $n$,  constructed from a Lagrangian disk $L$ as above, there is an extension of $\ev_1$ over $(\DD\setminus \bigcup_{j=1}^{m}\DD_j) \times \MM$ with image contained in $\phi(S^n) \cup \{x_1 \leq -\tau\}$.
\end{Lemma}

\begin{proof}
As indicated at the start of the section, we aim to construct an extension of $\ev_1$ using holomorphic disk components of broken curves (whose evaluation image we will be able to control, in contrast to that of disks in $\FF_1$, say).  At this stage, the key requirement will be to construct the first step of the extension so as to have image in $\phi(S^n) \cup\{x_1 \leq -\tau\}$.

For glued maps $w\colon D\to \R^{2n}$, where $w=u\#v=\Psi(u,v,\rho)$ for $(u,v)\in U \times \MM\subset \FF_{1}\times\MM$, the boundary point $1 \in \partial D$ at which we evaluate when defining $\ev_1$, $\ev_{1}(w)=w(1)$, lies in the part of the domain $D$ of the glued map that comes from its component $u$ in the factor $\FF_{1}$. Since the glued and pre-glued maps are arbitrarily close, this means that $\ev_{1}(w)=w(1)$ is arbitrarily close to $u(1)$, see Figure \ref{fig:evpoint}. 

Consequently, by Lemma \ref{l:safeeval}, if $w^{\pm}=u^{\pm}\# v$ for $(u^{\pm},v)\in \partial_{\pm} I_{j}\times\MM$, we can connect the two points $\ev_{1}(w^{\pm})\in V^{\pm\sigma}$ to $\ev_{1}(w^{\mp})\in V^{\mp\sigma}$ by a short straight line segment in $\R^{2n}$ near the double point. This gives an extension of the evaluation map to $\partial(\DD\setminus \bigcup_{j=1}^{m}\DD_j)\times\MM$, where we map $\pa(\DD_{j}\setminus I_{j})\times\MM$ to the line segments above. 

The union of $\phi(S^{n})$ and a small ball $B$ in $\R^{2n}$ around its double point is homotopy equivalent to the space obtained from $S^{n}$ by attaching a $1$-cell $e$ with one endpoint in $V^{+}$ and the other in $V^{-}$. The fundamental group of this space is generated by a loop $\gamma\cup e$, where $\gamma$ is a path in $S^{n}$ connecting the endpoints of $e$. We take $\gamma$ to lie in the part of $S^{n}$ mapped by $\phi$ to $\{x_1\le -\tau-\epsilon\}$ and think of $e$ as a short path in $B$. Then it is clear that $\{x_1\le -\tau\}$ contains a 2-cell bounding $\gamma\cup e$, which means that we can find the desired extension. To see this note that $\ev_{1}$ of the preglued map corresponding to $u\# v\in \partial(\DD\setminus \bigcup_{j=1}^{m}\DD_j) \times \MM$  is independent of $v$. We then first homotope the image of $\ev$, in $\phi(S^{n})\cup B$, into $\gamma\cup e$ and then extend over the two cell $(\DD\setminus \bigcup_{j=1}^{m}\DD_j) \times \MM$. The actual glued map is arbitrarily close to the preglued one, and existence of the desired extension follows. 
\end{proof}

Assume that an extension $\widetilde{\ev}_{1}$ over $(\DD\setminus \bigcup_{j=1}^{m}\DD_j)\times \MM$ as constructed in Lemma \ref{Lem:extension_first_part} has been fixed. Note that for this extension $\widetilde{\ev}_{1}(\partial (\DD\setminus \bigcup_{j=1}^{m}\DD_j)\times \MM)$ maps into a small neighborhood of the double point. We use the holomorphic disks which are parametrized by $\MM$ to extend the map to the remainder of the filling $\TT=\DD\times\MM$.   

\begin{Lemma} \label{Lem:extension_second_part}
For sufficiently large $\rho_0$, there exists a further extension of $\ev_1$ to all of $\DD\times\MM$ such that the image of each subspace $\DD_{j}\times \MM$ lies in an arbitrarily small neighborhood of the total evaluation map
	\[ 
	\ev\colon D \times \MM \longrightarrow \R^{2n}, \quad (z,u) \mapsto u(z),
	\] 
where $D$ denotes the source disk for maps in  (the gauge-fixed model of) $\MM$.
\end{Lemma}

\begin{proof}
	Fix $I_j$. As the location of the glued-in holomorphic disk sweeps $I_j$, the location of the boundary point $1$ on the glued domain sweeps the boundary of the once punctured source disk of maps in $\MM$. More formally, a glued disk $w=u\# v$ is close to breaking, which means that the restriction of $w$ to a small half disk around the point in $\pa D$ where $u$ has a negative puncture is arbitrarily close to $v$ for sufficiently large gluing parameter. This means that as the puncture of $u$ moves through the interval $I_j$ around $1$, the evaluation map $\ev_{1}(u\# v)$ is arbitrarily close to the evaluation map $v|_{\partial D}$, see Figure \ref{fig:evpoint}.
	   
	As mentioned above, by construction of the map over $\widetilde{\ev}_{1}$, its image $\widetilde {\ev}_{1}(\partial (\DD\setminus \bigcup_{j=1}^{m}\DD_j)\times \MM)$ is approximately constant and lies in a small ball around the double point of $\phi$. Contracting this ball to the double point (and the maps into it to constants), we may identify the products
\[
\DD_j \times \MM \quad \textrm{and}  \quad D \times \MM,
\]
	in such a way that the map $\ev_1|_{I_{j}\times\MM}$ corresponds to the evaluation along the boundary $\pa D$. We then naturally extend the map over $\DD_j$ as the evaluation map over $D$ under this identification. The lemma follows.
\end{proof}

\begin{figure}[ht]
	\centering
	\includegraphics[width=.6\linewidth]{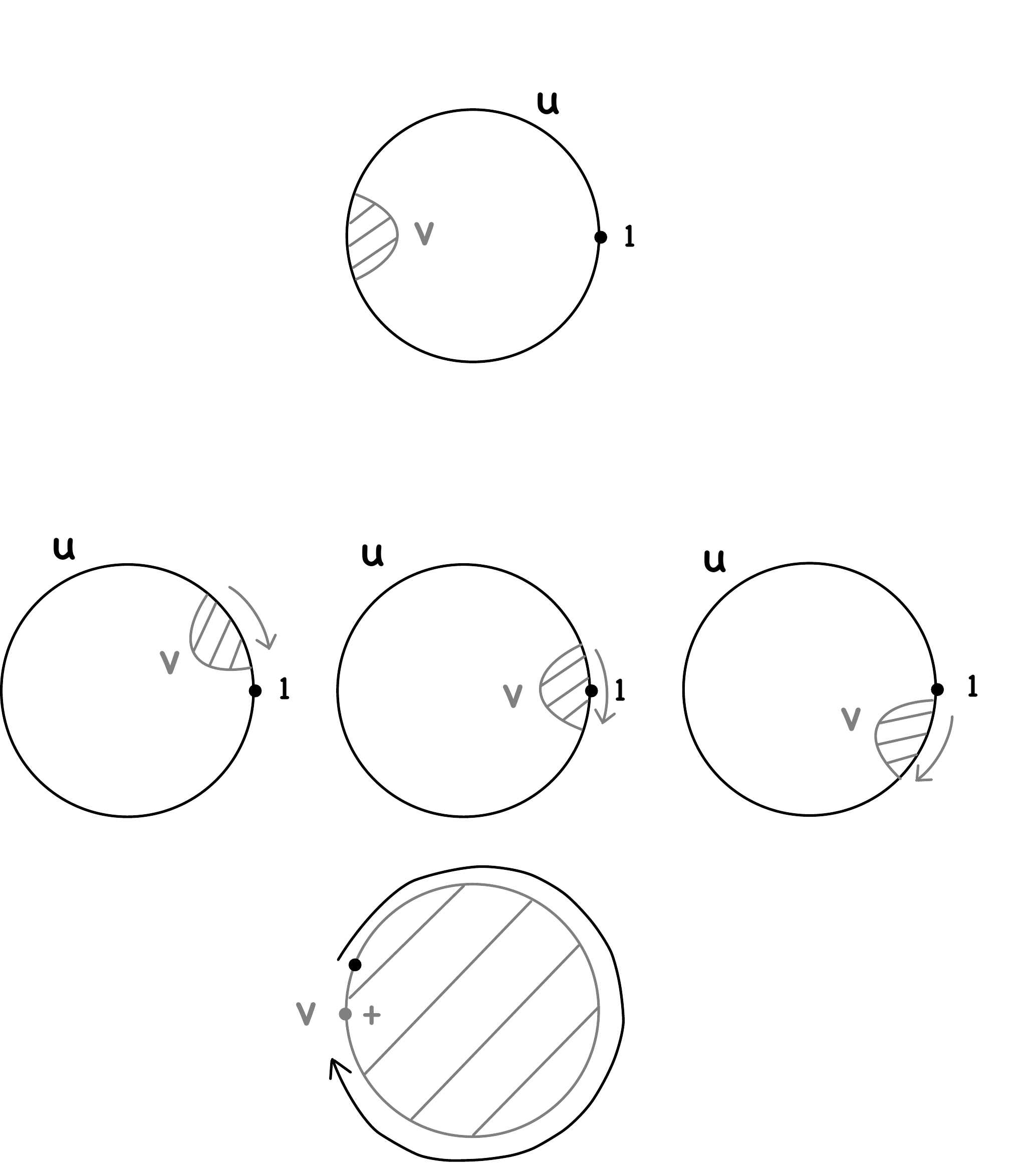}
	\caption{Top: the evaluation point $1$ is far from the glued in holomorphic disk $v$, as in Lemma \ref{Lem:extension_first_part}. Middle: the glued in holomorphic disk passes the evaluation point $1$, as in Lemma \ref{Lem:extension_second_part}. Bottom: the passage as viewed from the holomorphic disk, where the evaluation point sweeps the whole boundary. }
	\label{fig:evpoint}
\end{figure}

\section{Homotopy rigidity for Lagrangian disks}
In this section we prove Theorem \ref{Thm:yasha} and Corollary \ref{Cor:yasha} in two steps. We first compactify the Lagrangian disk $L$ with Legendrian boundary $\Lambda$ to an immersed Lagrangian sphere, the Lagrangian projection of $\Gamma(L)$, with double points in natural 1-1 correspondence with Reeb chords of $\Lambda$. Then we apply the results of Section \ref{Sec:modulispaces} for a specific almost complex structure on $\R^{2n}_{\st}$. 

\subsection{Compactification}\label{sec:comp}
Let $(x_1,y_1,\dots,x_n,y_n)$ be coordinates on $T^{\ast}\R^{n}$ and let $x\in\R^{n}$. Let $L$ be a Lagrangian disk agreeing with the fiber $T^*_{x}\R^{n}$ outside a compact set. We will first show how to change $L$ in a neighborhood of infinity to obtain a  Lagrangian disk with Legendrian boundary in the sense of Section \ref{Sec:leg_boundary}. The Legendrian boundary will in fact be the standard Legendrian unknot. 

To this end, we consider a Lagrangian disk $D\subset\R^{2n} = T^*\R^n$ with Legendrian boundary the standard Legendrian unknot, and such that $D$ intersects the zero-section in one point and agrees with the fiber at that point in a neighborhood of this intersection point.  In agreement with the conventions used after Lemma \ref{l:safeeval}, we assume that $D$ is a cone over a Legendrian unknot over the \emph{negative} $x_1$-axis.  Consider scaling the fiber coordinates in $T^{\ast}\R^{n}$ by $\lambda>0$. By choosing $\lambda$ sufficiently large we make the intersection of $D$ with an arbitrarily large ball agree with the fiber at the intersection point. Removing the intersection with the ball and inserting the corresponding part of $L$, we get the desired Lagrangian disk with Legendrian boundary equal to the standard unknot.   This disk agrees with $L$ in a ball, and is a cylinder over a Legendrian unknot in the region $\{x_1 \leq -\tau\}$ for sufficiently large $\tau>0$.

Let $S_k \subset \bR^n \setminus \{x\}$  be a point if $k=0$, or a $k$-sphere in $\R^{n}$ if $1\le k< n-1$, as in Section \ref{Sec:intro}. By a small homotopy, it will be convenient to normalize choices as follows (cf. the discussion after Lemma \ref{l:safeeval}). We take $x=0$ to be the origin, and we assume that $S_k \subset \{x_1=\varepsilon\}$ lies in an affine hyperplane (so in particular $x=0$ does not lie in the convex hull of $S_k$), for some small $\varepsilon > 0$. We have a Lagrangian disk $L$ that agrees with the fiber $T_0^*\R^n$ over $0$ outside a compact set, and which is disjoint from $T^{\ast}_{S_{k}}\R^{n}$. We alter $L$ as above, inserting a cone on the unknot over the negative $x_1$-axis, to obtain a Lagrangian disk with Legendrian boundary in the sense of Section \ref{Sec:leg_boundary}. 

%

Consider now the Lagrangian projection of the Legendrian sphere $\Gamma(L)\subset \R^{2n}\times\R$, as constructed in Section \ref{Subsec:half}. If we view the standard Legendrian unknot as living in the contact hypersurface $\{x_{1}= -\tau+\tau_{0}\}$, then $\Gamma(L)$ agrees with the cone over  that unknot in the region $\{-\tau\le x_1\le -\tau+\tau_0\}$, for some large $\tau>\tau_{0}>0$. Furthermore, since $L$ has Legendrian boundary the standard unknot, $\Gamma(L)$ has exactly one Reeb chord, which has grading $n$. The corresponding double point of its projection lies in $\{x_1= -\tau-\epsilon\}$, where we form $\Gamma(L)$ by the doubling construction centered at $\{x_{1}=-\tau-\epsilon\}$. Denote the resulting Lagrangian sphere immersion with one double point  by $\phi\colon S^{n}\to\R^{2n}$.

\begin{Lemma}\label{Lem:hmtpywhitney}
The map $\pr \circ \phi$ defines a map $S^n \to S^n \setminus S_k$ whose homotopy class agrees with the class $[f_L]$.
\end{Lemma}

\begin{proof} 
This follows for the same reasons as in the discussion at the end of Section \ref{Sec:leg_boundary}.  We have arranged that the map $\phi$, when restricted to the disk $\Gamma(L)\cap \{x_{1}\leq -\tau\}$, maps outside a large sphere containing $S_k$. (Note that although $\Gamma(L)$ is constructed by a doubling procedure, there is no doubling of the locus $S_k$, so the homotopy class reproduces $[f_L]$ and not twice that class.) 
\end{proof}

\subsection{Degenerating the almost complex structure\label{Sec:degenerate_1}}  We continue in the setting of Section \ref{sec:comp}. The map $\phi\colon S^n \to \R^{2n}_{\st}$ satisfies the conditions from the start of Section \ref{sec:basic}, so we have the moduli spaces of Floer-holomorphic and holomorphic disks as before.  Bearing in mind Lemma \ref{Lem:extension_first_part}, the evaluation map $\ev_1\colon \overline{\FF}_1 \to \phi(S^n)\subset \R^{2n}$ admits an extension over $(\DD\setminus \bigcup_j \DD_j)\times \MM$ with image inside the half-space $\{x_1 \leq -\tau\}$, which is disjoint from $\pr^{-1}(S_k) = (T^*\R^n)|_{S_k}$.  To construct the desired extension of the evaluation map to all of $\BB$ with image outside $(T^*\R^n)|_{S_k}$, in light of Lemma \ref{Lem:extension_second_part}, it suffices to prove that  for a suitable almost complex structure $J$ on $T^{\ast}\R^{n}$, the image of the total evaluation $\ev\colon D \times \MM\to T^{\ast}\R^{n}$ is completely disjoint from the submanifold $\pr^{-1}(S_k)\subset T^{\ast}\bR^{n}$. 

To that end, we consider a family of almost complex structures associated to the limit in which the fibres of $\pr\colon T^{\ast}\R^{n} \to \bR^n$ are shrunk to zero volume. Let $\beta \in \pi_2(T^{\ast}\R^{n},\phi(S^n))$ denote the generator of positive area, so for any taming almost complex structure $J$, the holomorphic discs in $\MM$  have relative homotopy class $\beta$ and area $A(\phi)=\int_{\beta}\omega_{\st}$.

\begin{Lemma}\label{Lemma:scale} 
	Suppose $J$ tames $\omega_{\st}$. 
	Let $\eta_\lambda\colon T^{\ast}\R^n \rightarrow T^{\ast}\R^n$ denote the map $\eta_\lambda(x,y) = (x,\lambda y)$. Then the area of any once punctured $J$-holomorphic disk with boundary on $\eta_\lambda(\phi(S^n))$ equals $A(\eta_{\lambda}\circ \phi)=\lambda A(\phi)$. 
	In particular  there exists $\lambda_0>0$ such that for any $0<\lambda<\lambda_0$ any $J$-holomorphic disk with one positive puncture and boundary on $\eta_\lambda\circ\phi$ is disjoint from $\pr^{-1}(S_{k})$.
\end{Lemma}

\begin{proof}
The area scaling is immediate. The second result follows from monotonicity for holomorphic disks.  The boundary of the disk lies at finite distance $\ge \epsilon/100$ from $\pr^{-1}(S_k)$. Hence if the disk passes through a point in $S_k$ it has area bounded below by $C \epsilon^{2}$, for some constant $C>0$ related to the taming condition. Take $\lambda_0 < C\epsilon/A(\phi)$. 
\end{proof}

Equivalently, one can fix the Lagrangian immersion $\phi(S^n)$, and consider the almost complex structures $J_\lambda = \eta_\lambda^*(J)$. Let $\lambda_0$ be as in Lemma \ref{Lemma:scale} and write $\MM_{\lambda}$ for the moduli space of $J_{\lambda}$-holomorphic disks with one positive puncture and boundary on $\phi(S^n)$.

%

\begin{Corollary} \label{Cor:discs_localise}
	For $0< \lambda<\lambda_0$, the image of the evaluation map $\ev\colon D \times \MM_{\lambda} \to T^{\ast}\R^{n}$ is disjoint from $\pr^{-1}(S_k)$.
\end{Corollary}

\begin{proof}
This follows from Lemma \ref{Lemma:scale} on observing that	
if $u$ is a $J_\lambda$-holomorphic disk with boundary on $\phi(S^n)$ then $\eta_\lambda\circ u$  is a $J$-holomorphic disk with boundary on $\eta_\lambda\circ \phi(S^n)$.  
\end{proof}

\begin{Remark}
	A more vivid route to proving Corollary \ref{Cor:discs_localise} is to recall the correspondence, for $t \gg 0$ sufficiently small, between $J_t$-holomorphic discs with boundary on the exact Lagrangian immersion $\phi(S^n)$, and Morse flow trees for the corresponding Lagrangian front $\Lambda = \pr\circ\phi(S^n)$.  By hypothesis, $S_k\subset \bR^n \setminus \Lambda$, which means that there are \emph{no sheets} of the Lagrangian $\phi(S^n)$ lying over $S_k$. In particular, no Morse flow-tree for the Lagrangian can approach $S_k$, which is just the conclusion of the preceding corollary.
\end{Remark}

\subsection{Proofs of the local results}\label{Subsec:Main_proof}

\subsubsection{Proof of Theorem \ref{Thm:yasha}}\label{ssec:pfthmyasha}
Lemma \ref{Lem:hmtpywhitney} shows that $f_L\colon S^{n}\to \R^{n}\setminus \{0\}$ is homotopic to the composition $\pr\circ \phi$ where $\pr\colon T^{\ast}\R^{n}\to \R^{n}$ is the bundle projection, and where the Lagrangian immersion $\phi\colon S^{n}\to T^{\ast}\R^{n}$ with one double point was constructed in Section \ref{sec:comp}. Proposition \ref{Prop:spin} then gives a spin $(n+1)$-manifold $\BB$ such that $\partial\BB= S^{n}$ (compare to the proof of Corollary \ref{Cor:yasha} where we use that $\BB$ is parallelizable) and Lemma \ref{Lem:extension_second_part} gives an extension of the map $\ev_1=\phi$ on $S^{n}$ to a map $\ev_{1}\colon \BB\to T^{\ast}\R^{n}\setminus T_{0}^{\ast}\R^{n}$. Composing with the projection $\pr$ we find that $\pr\circ\phi$ extends to $\pr\circ\ev_1\colon\BB\to\R^{n}\setminus\{0\}$. Consider now the preimage $C_{\ell}=(\pr\circ\phi)^{-1}(\ell)$ of a ray $\ell\approx [0,\infty)$ emanating at $0\in\R^{n}$ which is transverse to $\pr\circ\phi$. Then $C_{\ell}$ is a closed $1$-manifold, and the tangent space of the sphere of rays at $\ell$ gives a normal framing of $C_{\ell}$. Similarly, let $F_{\ell}\subset \BB$ denote the oriented surface which is the preimage $(\pr \circ \ev_1)^{-1}(\ell)$ of $\ell$, similarly equipped  with a normal framing.

Let $\xi_{S^{n}}$ denote the unique spin structure on $S^{n}$. We view $\xi_{S^{n}}$ as a class in $H^{1}(SO'(S^{n}))$, where $SO'(S^{n})$ denotes the bundle of oriented frames on the stabilized tangent bundle $TS^{n}\oplus\R$. Similarly, we view the spin structure on $\BB$ as an element $\xi_{\BB} \in H^1(SO(\BB))$, where $SO(\BB)$ denotes the frame bundle of $\BB$. Note that $\xi_{S^{n}}$ is the restriction of $\xi_{\BB}$ to the boundary $S^{n}$. 

Given a normal framing $\nu$ of a circle $\gamma$ in $S^{n} \subset \BB$ we equip it with a full framing of $T\BB|_{\gamma}$ by adding a framing $\tau$ in the two remaining directions which rotates once compared to the framing given by the tangent vector of $S^{1}$ followed by the constant vector in $\R$. Let $\vec{\gamma}$ denote the curve $\gamma$ lifted to $SO'(S^{n})$ by the framing $(\tau,\nu)$. An explicit calculation shows that $\nu$ is the null cobordant framing if and only if
\[ 
\langle\xi_{S^{n}},\vec{\gamma}\rangle=0.
\] 

For the link $C_{\ell}$ it is elementary to check that the framing $(\tau,\nu)$ extends to a framing of $F_{\ell}$. This means that $\vec{C_{\ell}}$ bounds in $SO(\BB)$. Hence
\[ 
\langle\xi_{\BB},\vec{C_{\ell}}\rangle=\langle\xi_{S^{n}},\vec{C_{\ell}}\rangle=0.
\] 
This shows that $C_{\ell}$ is framed null-cobordant, which by the Pontryagin-Thom construction finishes the proof.
\qed

\subsubsection{Proof of Corollary \ref{Cor:yasha}\label{Subsec:Pf_yasha}}
The proof of Corollary \ref{Cor:yasha} follows similar lines, but uses triviality of the tangent bundle of the filling $\BB$ beyond its spin structure.

Let $L$ be the Lagrangian disk. The stable homotopy class of $f_{L}$ is the homotopy class of the $m$-fold suspension $\Sigma^{m}f_{L}$ for any sufficiently large $m$. By Lemma \ref{Lem:suspension}, this is also the homotopy class of the map $f_{C^{m}(L)}$ associated to the $(n+m)$-dimensional Lagrangian disk $C^{m}(L)= C(C(\dots C(L))$. We point out that although the Legendrian boundary of $C^{m}(L)$ might not be the standard Legendrian unknot, it is a knot with only one Reeb chord, and that is sufficient for the argument below to apply.

Picking $m$ sufficiently large to be in the stable range, and also so that $\pi_{n+m}(U/O)$ vanishes (e.g. $n+m \equiv -1 \mod 8$), Proposition \ref{Prop:parallel} gives a parallelizable manifold $\BB$ with boundary $S^{n+m}$ and Lemma \ref{Lem:extension_second_part} gives an extension $\pr\circ\ev_1$ of the corresponding map $\phi\colon S^{n+m} \to T^*(\R^{n+m}\setminus S^k)$ to $\BB$. Consider the composition $\pr\circ\phi$ followed by the inclusion $\R^{n+m}\subset S^{n+m}$. Let $\Gamma \approx S^{n+m-k-1} \subset \R^{n+m}\setminus S^k$ be a fiber sphere in the unit normal bundle of $S^{k}$. Then $S^{n+m}\setminus S^{k}\approx \Gamma\times D^{k+1}$ is foliated by open $(k+1)$-disks $D_{\gamma}$ parameterized by $\gamma\in\Gamma$. For $D_{\gamma}$ transverse to $\pr\circ \ev_{1}\colon\BB\to S^{n+m}\setminus S^{k}$ we find that 
\[ 
M_{\gamma}=(\pr\circ\ev_{1})^{-1}(D_{\gamma})\subset S^{n+m}
\]
is a framed submanifold that framed bounds the framed submanifold
\[ 
W_{\gamma}=(\pr\circ\ev_1)^{-1}(D_{\gamma})\subset\BB.
\]
The manifold $\BB$ is a parallelizable $(n+m+1)$-manifold; the stability condition $k<(n+m-3)/2$ guarantees that we can perform framed surgery on $\BB$, along framed spheres in the interior and disjoint from $W_{\gamma}$, to yield a parallelizable manifold $\BB'$ which is  $(k+2)$-connected. Fix a Morse function on $\BB'$ with no critical points of index $\leq k+2$, and with gradient field pointing outwards along the boundary.  Via  gradient flow, one can then isotope the $(k+2)$-dimensional subset $W_{\gamma} \subset \BB'$ into a collar neighborhood of the boundary of $\BB'$ (which is identified with a collar neighborhood of the boundary of $\BB$), since for dimension reasons it will generically miss the ascending manifolds of all critical points. By the Pontryagin-Thom construction the result again follows. \qed

\begin{Remark}\label{Rmk:3dim}
In the borderline case $\{k=0, n=3\}$, the relevant homotopy group $\pi_3(S^2)$ is not stable, but Corollary \ref{Cor:yasha} shows vanishing of the corresponding stabilized invariant.  Concretely this means that
if $L \subset T^*(\R^3\backslash \{0\})$ is a Lagrangian disk which coincides with a fiber outside a compact set, then $[f_L] \in \pi_3(S^2) \cong \Z$ is even.  

Note that when $n=3$ the moduli space $\FF_{-2}$ has virtual dimension zero, cf. Lemma \ref{Lem:vdim_floer}, so the construction of the bounding manifold $\BB$ does not go through directly in this case.
\end{Remark}

\subsection{Quasi-isomorphism type}\label{Subsec:quasi-iso}

For a closed exact Lagrangian $L \subset T^*Q$, it is known that $L$ is isomorphic to the zero-section in the Fukaya category.  We point out that the hypotheses in Theorem \ref{Thm:yasha} imply the analogous result for the nearby Lagrangian fiber, though we know of no Floer-theoretic argument to  constrain the homotopy class $f_L$.

Fix a coefficient field $\bK$.  A Liouville like manifold $(Y, \omega=d\theta)$ (see \cite[Appendix B.3]{EL} for a more detailed description of manifolds $Y$ as considered below) then has a well-defined wrapped Fukaya category $\scrW(Y)$, an $A_{\infty}$-category over $\bK$, whose objects are exact spin Lagrangian submanifolds  which are either closed or are cylinders on Legendrian submanifolds of the ideal boundary near infinity; see \cite{Abouzaid-Seidel} for the construction.  

\begin{Lemma}\label{Lem:object_in_wrapped}
Suppose $n\geq 3$. In the situation of Theorem \ref{Thm:yasha}, $L$ defines an object of the wrapped Fukaya category $\scrW(T^*(\bR^n\setminus \{0\}))$ quasi-isomorphic to the cotangent fibre.  
\end{Lemma}

\begin{proof} We compactify the base $\bR^n\setminus \{0\} = S^{n-1} \times \bR$ to $S^{n-1} \times S^1$, and we compactify the distinguished fibre $T_x^*$ by adding a handle to the Legendrian unknot at infinity. This yields the plumbing $X = T^*(S^{n-1}\times S^1) \#_{\partial} T^*S^n$, which is naturally a Liouville domain with an exact symplectic structure $\omega$.  The manifold $X$ contains non-compact Lagrangian thimbles $T_p^*$ and $T_q^*$, which are cotangent fibres to points $p\in S^{n-1}\times S^1$ and $q\in S^n$, in both cases lying near infinity (i.e. in the parts of the base added in the compactification).  The Lagrangian disk $L$ extends to a Lagrangian sphere $L'$ in $X$, which meets the cotangent fibre $T_q^*$ transversely once, and is disjoint from the fibre $T_p^*$. Since $n\geq 3$, the results of \cite{AbSm} show that every compact Lagrangian submanifold in $X$ can be expressed as a twisted complex on the two core components, and then the conditions
\[
HF(L', T_p^*) = 0, \quad HF(L',T_q^*) = \bK
\]
imply that $L'$ is quasi-isomorphic in the wrapped (and hence compact) Fukaya category $\scrW(X)$ to the core component $S^n \subset X$.  Applying Viterbo functoriality \cite{Abouzaid-Seidel}, we infer that the original Lagrangian $L$ is quasi-isomorphic to the cotangent fibre of $\bR^n\backslash \{0\}$. \end{proof}

\section{Whitney spheres in the complement of a monotone Lagrangian}
In this section we consider a more global version of Theorem \ref{Thm:yasha}, where we replace the fiber $T_0^*\bR^n$ by a monotone submanifold $C$ which is either closed or has ideal Legendrian boundary $\Gamma$, and study the homotopy class of an immersed Lagrangian sphere $S$ with one double point in the complement of $C$. As in section \ref{Sec:degenerate_1}, we will study the moduli space $\BB$ for a specific almost complex structure for which the holomorphic disks with boundary on $S$ do not intersect $C$. The appropriate almost complex structure is obtained from symplectic field theory, stretching the boundary of a neighborhood of $C$. In the limit holomorphic curves then fall apart into holomorphic buildings, and dimension estimates for the pieces yield sufficient conditions for holomorphic disks on $S$ to stay away from $C$. 

Note that, for a general Lagrangian submanifold $C$, there is no global projection to a sphere linking $C$, as existed in the special case for $T_0^{\ast}\R^{n}$ considered previously, and hence we cannot associate the linking homotopy class via such a projection. We refer to Appendix \ref{sec:algtop} for a general discussion of the topology of the complement of a Lagrangian submanifold in $\R^{2n}$. In Sections \ref{Sec:index1} and \ref{sec:stretch} we give dimension formulae for holomorphic curves and then turn to a more detailed study of the holomorphic buildings that arise. This leads to a proof of Theorem \ref{Thm:embLag}, which occupies Section \ref{Sec:pf_second_thm}.

\subsection{Neck stretching around Lagrangian submanifolds}\label{sec:neckstretch}
Let $C$ be a closed Lagrangian submanifold of $\R^{2n}_{\st}$. Then $C$ has a neighborhood which is symplectomorphic to a disk bundle neighborhood of the zero-section in $T^{\ast}C$. We identify this neighborhood with the union of two pieces: a collar region $[0,-1)\times U^{\ast}C$, where $U^{\ast}C$ denotes the unit cotangent bundle equipped with the standard contact form and the product has the symplectic form of the symplectization,  together with a suitably scaled version of the unit disk cotangent bundle $D^{\ast}C$ glued in. Stretching the neck then corresponds to replacing the collar region by $[0,-T)\times U^{\ast}C$ for $T\to\infty$. In the limit the manifold separates into two pieces, one a completion of $\R^{2n}_{\st} \setminus C$ with negative end $U^{\ast}C$, and the other symplectomorphic to $T^{\ast} C$ with positive end $U^{\ast}C$. 

Given instead a Lagrangian submanifold $C$ with non-empty ideal Legendrian boundary $\Gamma$, before stretching the neck around $C$ we must adjust the Liouville structure, so the Liouville vector field agrees with the Liouville field of the cotangent bundle of $C$ near $C$. To accomplish this we proceed as in \cite[Section 6.1 \& Figure 7]{ENS}. Consider a large ball such that $C$ intersects its boundary sphere in $\Gamma$. We create the new Liouville structure by attaching $D^*(\Gamma \times [0,\infty))$ along $\Gamma$, where $D^*\Gamma$ denotes a disk subbundle of the cotangent bundle $T^{\ast}\Gamma$. One can interpolate between the usual radial Liouville vector field on $\R^{2n}_{\st}$ and the field $p\cdot \partial_p$ in the cotangent bundle neighborhood to define a Liouville field for the new structure. 
This gives a Liouville-like manifold, that we denote $\R^{2n}_{C}$, with non-compact contact boundary. 
Note that the Liouville-like manifolds we consider look like ordinary Liouville manifolds except that they have additional ends where they look like the standard fiberwise Liouville structure on the cotangent bundle $T^{\ast}(\Gamma\times[0,\infty))$ where $\Gamma$ is a closed $(n-1)$-manifold. We refer to
\cite[Appendix B.3]{EL} for a discussion of such domains.  It still contains $C$ as an embedded Lagrangian submanifold, and in $\R^{2n}_C$ we can stretch the neck around $C$ just as in the closed case. The Maslov class of $C$ is unaffected by the passage from $\R^{2n}_{\st}$ to $\R^{2n}_C$, and in particular monotonicity of $C$ is unaffected.  For uniform notation, we let $\R^{2n}_{C}=\R^{2n}_{\st}$ when $C$ is closed.  

\begin{Remark}
Typically in SFT one stretches the neck around a closed contact (or stable Hamiltonian) hypersurface. When $C$ has non-empty ideal boundary, the contact boundary of $\R^{2n}_C$ is not compact, but monotonicity constrains all holomorphic curves we consider to live in a compact set, where the contact hypersurface is geometrically bounded. This is sufficient for the conclusions of \cite{BEHWZ} to apply.
\end{Remark}

\subsection{Conley-Zehnder and Morse indices\label{Sec:index1}}

Virtual dimensions of moduli spaces of punctured holomorphic curves are typically expressed in terms of Conley-Zehnder indices of Reeb orbits. These may be canonically defined in the unit cotangent bundle of an orientable manifold,  but  depend on  additional choices in the non-orientable case.  In this section we explain our convention for such indices.

Let $C\subset \R^{2n}_{\st}$ be a Lagrangian submanifold. 
Fix a Riemannian metric on $C$. The Reeb orbits on the unit cotangent bundle are the oriented closed geodesics for the metric. We will  not distinguish the geodesic from the Reeb orbit in our notation. In case $C$ has Legendrian boundary $\Gamma$ we choose the metric in the cylindrical end $[0,\infty)\times \Gamma$ to have the form $ds^{2}=dt^{2}+ f(t) g$, where $g$ is a metric on $\Gamma$ and where $f(t)>0$ is a function with $f'(t)>0$. Then the $t$-coordinate along a non-constant geodesic in the cylindrical end cannot have a local maximum. Hence, there are no closed geodesics in the end. 

If $\pr\colon T^*C \to C$ is projection, then $T(T^*C) \cong \pr^*(TC\otimes \C)$. 
Let $\det_{\C}(TC)$ denote the corresponding determinant bundle, i.e. the complex line bundle which is the determinant of the complexified tangent bundle $TC \otimes \C$.

\begin{Lemma}\label{l:dettrivial}
The bundle $TC\otimes\C$ is trivial and hence the complex line bundle $\det_{\C}(TC)$ is trivial.
\end{Lemma}

\begin{proof}
Consider the standard complex structure $i$ on $\C^{n}\approx\R^{2n}_{\st}$. Since $C$ is Lagrangian, multiplication by $i$ gives an isomorphism from the tangent bundle of $C$ to its normal bundle in $\C^{n}$. It follows that $TC\otimes \C\approx T\C^{n}|_{C}$. The lemma follows. 	
\end{proof}

We next define trivializations of $\det_{\C}(TC)$. Consider the real line bundle $\det(TC)$ over $C$ and fix a section $s\colon C\to \det(TC)$ that is transverse to the $0$-section. (When $C$ is orientable we can take $s$ to be non-vanishing). Let $P=s^{-1}(0)$. 

\begin{Lemma} \label{lem:P_orientable}
 The zero-set $P\subset C$ is an orientable submanifold whose $\Z/2$-Poincar\'e dual  represents the first Stiefel-Whitney class $w_1(TC) \in H^1(C;\Z/2)$. An open neighborhood $N(P)$ of $P$ in $C$ is also orientable, so $N(P) \approx P\times(-\epsilon,\epsilon)$. A choice of  orientation of the normal bundle of $P$ in $C$ defines an integral lift $W_1 \in H^1(C;\Z)$ of $w_{1}(TC)$.
 \end{Lemma}
 
 \begin{proof} 
 First note that the total space $E$ of $\det(TC)$ is orientable. The normal bundle to $P$ in $E$ is the sum of the normal bundle $\nu_{1}$ to $P$ in $C$ and the normal bundle $\nu_{2}$ of $C$ in $E$. Since $P$ is Poincar\'e dual to $w_{1}(TC)$ it follows that $w_{1}(\nu_{1})=w_{1}(TC)|_{P}$. Since $\nu_{2}=\det(TC)$ by definition of $E$, also $w_{1}(\nu_2)=w_{1}(TC)|_{P}$. Hence the normal bundle to $P$ in $E$ has vanishing first Stiefel-Whitney class and since $E$ is orientable, $P$ is orientable.  
 
 
We claim that a neighborhood $N(P)$ of $P$ is orientable as well. To see this, note that, by orientability of $P$, $N(P)$ is orientable provided the normal bundle $\nu_{1}$ of $P$ in $C$ is trivial.  Let $\sigma\colon P\to \nu_{1}$ be a section of $\nu_{1}$ transverse to the $0$-section. Since $P$ is $\Z/2$-Poincar\'e dual to the first Stiefel-Whitney class $w_{1}(TC)$, $\sigma^{-1}(0)$ is Poincar\'e dual to $w_1(TC)^{2}$. But 
\[
 w_2(TC \otimes \C) = w_{2}(TC\oplus TC) = w_1(TC)^2 = 0,
\]
by Lemma \ref{l:dettrivial}. Hence $\nu_{1}$ is orientable and so is $N(P)$, and we have a tubular neighborhood $N(P)\approx P\times(-\epsilon,\epsilon)$ of $P$ in $C$ as required. Finally, picking an orientation of $\nu_{1}$ allows us to define a $\Z$-valued intersection number between loops in $C$ and $P$. Since $P$ is $\Z/2$-Poincar\'e dual to $w_{1}(TC)$, it is clear that the resulting class in $H^{1}(C;\Z)$ is an integral lift of $w_{1}$.
\end{proof}

To normalize sections in the construction below we use the metric on $TC$ to induce a metric $|\cdot|$ on the real line bundle $\det(TC)$. Let $s\colon C\to \det(TC)$ be the section above and define 
\[ 
v'(q)=\frac{1}{|s(q)|}s(q),\quad q\in C\setminus N(P).
\]
Then $v'$ trivializes $\det_{\C}(TC)$ over $C\setminus N(P)$ by viewing $v'$ as a non-zero section of $\det_{\C}(TC)$. 

By the above discussion also $\det(TC)|_{N(P)}$ is trivial. The boundary of $N(P)\approx P\times(-\epsilon,\epsilon)$ is disconnected and the restriction of $v'$ to the boundary component $\{-\epsilon\}\times P$ gives a unit length section. Let $v''\colon N(P)\to\det(TC)$ be the unique extension of this section as a unit length section then $v''$ gives a trivialization of $\det_{\C}(TC)$ over $N(P)$. The trivializations $v'$ and $v''$ can be compared along $P\times\{\pm\epsilon\}$. Since $s$ has a transverse zero along $P$ we find the following: $v'=v''$ along $P\times\{-\epsilon\}$ and $v'=-v''$ along $P\times\{\epsilon\}$. We use this to fix a trivialization $v$ of $\det_{\C}(TC)$ over all of $C$ as follows:
\begin{equation}\label{eq:v_triv} 
v=
\begin{cases}
v'(q) &\text{ for }q\in C\setminus N(P),\\
e^{i\pi\frac{(t+\epsilon)}{2\epsilon}}v''(p,t) &\text{ for }(p,t)\in P\times[-\epsilon,\epsilon]=N(P).
\end{cases}
\end{equation}
We will denote the Conley-Zehnder index of a Reeb orbit $\gamma$ defined with respect to this trivialization $\CZ^{w_1}$. Denote by $\iota(\gamma)$ the Morse index of an oriented closed geodesic, viewed as a critical point of the energy function on the free loop space. 

\begin{Lemma}\label{l:CZw_1}
The Conley-Zehnder index of a Reeb orbit $\gamma$ in $U^{\ast} C$ with respect to the trivialization $v$ of $\det_{\C}(TC)$, see \eqref{eq:v_triv}, satisfies the following:  
	\begin{equation} \label{eqn:alternative}
	\CZ^{w_1}(\gamma) = \iota(\gamma)-W_1(\gamma).
	\end{equation}
	\end{Lemma}
	
\begin{proof}
	The Conley-Zehnder index $\CZ^{w_{1}}$ measures rotations in the determinant of the linearized Reeb flow with respect to the given trivialization $v$ in \eqref{eq:v_triv}. Here the linearized Reeb flow is the linearized geodesic flow. After small deformation of $P$, we can assume that no Jacobi field of any closed geodesic vanishes near $P$, and that the closed geodesics intersect $P$ transversely. Taking $N(P)$ sufficiently small we then find that the linearized Reeb flow is  approximately constant with respect to $v''$ in $N(P)$ and thus rotates $-\pi$ relative to the reference trivialization $v$ of $\det_{\C}(TC)$ here. In $C-N(P)$ $v=v'$ is the standard section of the real determinant $\det(TC)$ and zeros of Jacobi fields give contributions to the Conley-Zehnder index. Adding these contributions proves \eqref{eqn:alternative}. 
\end{proof}
	

To compute dimensions of disks and spheres in the component $\R^{2n}\setminus C$ after stretching,  we use the Conley-Zehnder index of Reeb orbits as defined with respect to the standard global trivialization of the tangent bundle of $\R^{2n}_C$. More precisely, if $(x_{1},y_{1},\dots,x_{n},y_{n})$ are standard symplectic coordinates on $\R^{2n}$ then we think of $T\R^{2n}$ as a complex vector bundle $T'\R^{2n}$ trivialized by $(\partial_{x_{1}},\dots,\partial_{x_{n}})$. In a neighborhood of $C$, the complex line bundle $\det_{\C}(TC)$ is isomorphic to $\det(T'\R^{2n})$ and
\begin{equation}\label{eq:v_0}
v_{0}=\partial_{x_{1}}\wedge\dots\wedge\partial_{x_{n}}
\end{equation}
gives a non-zero section of $\det_{\C}(TC)$. We denote the Conley-Zehnder index of Reeb orbits defined with respect to the trivialization $v_{0}$ by $\CZ$. The following result relates $\CZ$ and $\CZ^{w_{1}}$; it is closely related to \cite[Theorem 3.1]{Viterbo}.  

\begin{Lemma}
	If $\gamma$ is a Reeb orbit in $U^{\ast}C$, if $\CZ$ denotes the Conley-Zehnder index defined with respect to the trivialization $v_{0}$, and if $\CZ^{w_{1}}$ is as in Lemma \ref{l:CZw_1}, then 
	\[ 
	\CZ(\gamma) = \CZ^{w_1}(\gamma)+ W_1(\gamma) - \mu_C(\gamma)=
	\iota(\gamma)-\mu_{C}(\gamma),
	\] 
	where $\mu_{C}(\gamma)$ denotes the Maslov index of the projection of $\gamma$ into $C$. 
\end{Lemma}

\begin{proof}
	We need to compare the trivializations $v$ in \eqref{eq:v_triv} and $v_{0}$ in \eqref{eq:v_0} of $\det_{\C}(TC)$ along a given Reeb orbit $\gamma$. Let $\ell_{v}$ and $\ell_{v_{0}}$ denote the fields of real lines in $\det_{\C}(TC)$ spanned by $v$ and $v_{0}$, respectively. To simplify the comparison we use the fact that Conely-Zehnder and Maslov indices are invariant under homotopies and replace $C$ by its image $\rho_{s}(C)$ under fiber scaling of $\R^{2n}$:
	\[ 
	\rho_{s}(x_{1},y_{1},\dots,x_{n},y_{n})= (s^{-1} x_{1}, s y_{1},\dots,s^{-1}x_{n},s y_{n}), \quad s>0.
	\]
	Note that $\phi_{s}^{\ast}\omega_{0}=\omega_{0}$, and that, since $\gamma$ lies in a compact subset of $C$, there is a neighborhood $U$ of $\gamma$ in $C$ such that $\rho_{s}(U)$ lies in an arbitrarily small neighborhood of the subspace $y_{j}=0$, $j=1,\dots,n$, provided $s>0$ is sufficiently small.   
	
	Let $\mathrm{Gr}_{\mathrm{Lag}}(\R^{2n})$ denote the Lagrangian Grassmannian of $\R^{2n}$. Consider the Lagrangian subspaces $\Pi_y=\{x_{j}=0, \ j=1,\dots,n\}$ and $\Pi_{x}=\{y_{j}=0, \ j=1,\dots,n\}$ of $\R^{2n}$, and let $Z \subset \mathrm{Gr}_{\mathrm{Lag}}(\R^{2n})$ denote the Maslov cycle of subspaces that intersect $\Pi_y$ in a subspace of dimension $>0$. After small deformation of $P$, we may assume that the path of Lagrangian tangent planes $TC|_{\gamma}$ along $\gamma$ does not intersect $Z$ in $\gamma\cap P$. Assume now that the scaling parameter $s>0$ is sufficiently small so that $\phi_{s}(TC)$ is approximately equal to the subspace $\Pi_x$ outside the caustic $\Sigma$ of $\phi_{s}(U)$ (the locus of tangencies with the subspace $\Pi_y$). Then each transverse intersection of $TC|_{\gamma}$ with $Z$ of sign $\epsilon=\pm 1$ corresponds to a transverse passage of $\ell_{v}$ through $\ell_{v_{0}}$ of sign $\epsilon$. To see this, note that such an intersection corresponds to a point where the curve $\gamma$ crosses the caustic $\Sigma$ transversely and the tangent space of $TC$ along $\gamma$ is approximately constant in the $(n-1)$ directions tangent to $\Sigma$ and makes a $\pi$-rotation in the complex line normal to $\Sigma$.
	
	Furthermore, the line $\ell_{v''}$, see \eqref{eq:v_triv}, is approximately constant with respect to $\ell_{v_{0}}$ in $N(P)\cap \gamma$, since $\gamma\cap P\cap Z=\varnothing$, and hence $\ell_{v}$ makes a $\pm\pi$-rotation with respect to $\ell_{v_0}$ for each intersection $P\cap\gamma$. Since $P$ is dual to $W_{1}$, adding these two contributions we find that
	\[ 
	\CZ(\gamma) - \CZ^{w_1}(\gamma)= W_1(\gamma) - \mu_C(\gamma),
	\]  
	as claimed.
\end{proof}

\subsection{Dimension formulae and neck-stretching}\label{sec:stretch}
Let $C\subset \R^{2n}_{\st}$ be a Lagrangian submanifold with Legendrian boundary $\Gamma$ and let $\phi(S^n)\subset \R^{2n}_{\st}$ be a Lagrangian sphere with exactly one double point $a$ of grading $|a|=n$. Since $\phi(S^{n})$ is compact, by choosing a sufficiently large ball, we consider $\phi$ as a Lagrangian immersion into $\R^{2n}_{C}$, see Section \ref{sec:neckstretch}. The moduli space $\MM$ of holomorphic disks in $\R^{2n}_{C}$ with boundary on $\phi(S^n)$ and one positive puncture at $a$ is a closed manifold of dimension 
\[ 
\dim(\MM)=|a|-1=n-1.
\]

Below we will consider the behavior of holomorphic disks in this moduli space under neck-stretching around $C$. More precisely, we will take limits of curves as we degenerate the almost complex structure near $C$ in such a way that, in the limit, the ambient $\R^{2n}_{C}$ falls apart in two pieces. The upper piece $X^{+}$ has one positive end as in $\R^{2n}_{C}$ and one negative end, the negative half of the symplectization of the unit conormal bundle $U^{\ast}C$. The lower piece $X^{-}=T^{\ast}C$. SFT-compactness \cite{BEE} describes the limits of holomorphic curves in $\R^{2n}_{C}$ in terms of holomorphic buildings with at least one level in $X^{+}$ and perhaps some levels in $X^{-}$, where the levels are asymptotic to Reeb orbits in $U^{\ast}C$. (The compactness theorem in \cite{BEE} also gives additional levels in the symplectization $\R\times U^{\ast} C$. Here we will simply consider such levels as curves in $X^{-}$, which means that the levels in $X^{-}$ in general consist of broken curves. Since the relevant indices and virtual dimensions are additive, this elision will not lose track of any information that we need.)

Since the pieces of the limiting holomorphic building glue to a disk, each piece is a disk or sphere with punctures. We will thus consider dimensions of the following curves:
\begin{itemize}
	\item Spheres in $T^{\ast}C$ with positive punctures at Reeb orbits in $U^{\ast}C$.
	\item Spheres in $X^{+}$ with negative punctures at Reeb orbits in $U^{\ast} C$.
	\item Disks in $X^+$ with a positive boundary puncture at $a$ and negative interior punctures at Reeb orbits in $U^{\ast}C$.
	\end{itemize}

We consider first a sphere $v$ in $T^{\ast}C$ with positive punctures at Reeb orbits $\gamma_1,\dots,\gamma_k$ of indices $\CZ^{w_1}(\gamma_1)=i_1$, \dots, $\CZ^{w_1}(\gamma_k)=i_k$. We have \cite[Section 1.7]{EGH}
\begin{equation}\label{eq:diminsidesphere}
\dim(v)= 2(n-3) +\sum_{j=1}^{k} (i_j -(n-3)).
\end{equation}

Consider next the case of a sphere $w$ in $X^+$ with no positive puncture and with negative punctures at Reeb orbits $\beta_1,\dots,\beta_r$. We have
\begin{align}\label{eq:end}
\dim(w) &= 2(n-3) - \sum_j(\CZ(\beta_j)+(n-3))\\\notag
&= 2(n-3) - \sum_j(\iota(\beta_j)-\mu_L(\beta_j)+(n-3)).
\end{align}

Similarly, if $u$ in $X^+$ is the component with positive puncture at $a$ and negative punctures at $\beta_1,\dots,\beta_r$ then
\begin{align}\label{eq:top}
\dim(u) &= (n-1) - \sum_j(\CZ(\beta_j)+(n-3))\\\notag
&= (n-1) - \sum_j(\iota(\beta_j)-\mu_L(\beta_j)+(n-3)).
\end{align}

Any punctured holomorphic curve has an underlying somewhere injective curve; for the latter result in the punctured setting see e.g. \cite{DRE}. Standard transversality arguments, perturbing the almost complex structure near an injective point, imply that moduli spaces of somewhere injective holomorphic curves are cut out transversely for generic data. In the proof of Lemma \ref{Lemma:wrongdimension} below, we use such transversality arguments only for curves where monotonicity of the Lagrangian implies that any multiple cover has higher virtual dimension than that of the underlying simple curve. 

Let $\beta \in \pi_2(\R^{2n},\phi(S^n))$ denote the generator of positive area.  As usual, fix an almost complex structure $J$ on $\C^n$ standard in a neighborhood of $\phi(S^n)$.

\begin{Lemma}\label{Lemma:wrongdimension}
	Let $C\subset \R^{2n}_{\st}$ be monotone with minimal Maslov number $\ge 3$. Let $J_t$ denote the sequence of almost complex structures on $\R^{2n}_{C}$ obtained by stretching the neck with parameter $t$ around $C$. If for some Riemannian metric on $C$ the minimal non-zero Morse index of a contractible geodesic loop  is $\geq 3$, then for  $t \gg 0$, no  $J_t$-holomorphic disk with one positive puncture and with boundary on $\phi(S^n)$ in homotopy class $\beta$ intersects $C$.
\end{Lemma}

\begin{proof}
	 The dimension of any curve in the limit is non-negative by transversality, and those  dimensions sum to $n-1$, the dimension of the space of disks in class $\beta$, by additivity of dimension in holomorphic buildings.  If no component of the limit building lies in $T^*C$ the conclusion is immediate, so assume some component is contained in $T^*C$.
	
	The dimension of any sphere with only one positive puncture in $T^{\ast}C$ is
	\[ 
	(n-3) + i > (n-1), 
	\] 
	where $i$ is the Morse index of the contractible geodesic corresponding to the Reeb orbit at the positive puncture.  There must therefore be a sphere with several positive punctures $\gamma_0,\dots,\gamma_m$, $m\ge 1$, inside $T^*C$; the dimension $d_{\rm in}$ of such a sphere is
	\[ 
	d_{\rm in}=2(n-3) + \sum_{j=0}^m (i_{j}-(n-3)).
	\]
	At each $\gamma_j$ there is (a possibly broken) outside sphere with a negative puncture at $\gamma_j$ of dimension $d_{\rm out}^{j}$ given by
	\[ 
	d_{\rm out}^{j}=(n-3)-i_j+m_j,
	\]
	where $m_j$ is the Maslov index of the geodesic corresponding to $\gamma_j$. Gluing the half cylinder on this geodesic with boundary on the zero section $C$ to the punctured sphere in $\R^{2n}_{C}-C$, we obtain a holomorphic disc with boundary on $C$. Monotonicity and the condition on minimal Maslov number then implies that $m_j\geq 3$. Assume now that the dimension $d_{\rm in}\ge n-2$. Then $0\le d_{\rm out}^{j}\le 1$, with right equality for at most one $j$. In particular,
	\[ 
	i_{j}-(n-3)\ge 2, \quad \textrm{with} \quad  i_j - (n-3) \geq 3 \quad \textrm{except  for  one  } j.
	\] 
	Thus 
	\[ 
	d_{\rm in}\ge (n-3)+i_0 + \sum_{j>0} (i_j-(n-3))\ge n,
	\]
	which contradicts $d_{\rm in}\le n-1$.  We conclude that $d_{\rm in} < n-2$, and hence the subset swept by evaluation of holomorphic curves in all  moduli spaces of building components  inside $T^*C$ has total dimension $< n$.  For generic data, this will then be disjoint from the zero-section $C\subset T^*C$.
\end{proof}

\subsection{Proof of Theorem \ref{Thm:embLag}\label{Sec:pf_second_thm}}

Fix $C \subset \R^{2n}_{\st}$ as in the formulation, adjust the Liouville structure near the ideal boundary of $C$ if necessary, and let $X = \R^{2n}_{C}\backslash C$.  From Lemma \ref{lem:pi_n}, and its analogue in the case when $C$ has boundary as discussed in Section \ref{sec:withboundary}, we know that $\pi_n(X)$ is determined as follows:
\[\pi_n(X) = \begin{cases} \Z/2\Z \oplus H_n(X;\Z) & w_2(C) = 0; \\ 
H_n(X;\Z) & w_2(C) \neq 0. \end{cases} \]
		Consider the completed moduli space $\BB$ of disks on $\phi(S^n)$ and the evaluation map $\ev\colon \BB\to \R^{2n}_{C}$. For an almost complex structure sufficiently stretched around $C$, the map $\ev$ takes $\BB$ into $X$. This immediately implies that $\phi(S^n)$ is nullhomologous in $X$, which completes the proof in the case that $w_2(C)\neq 0$.  

If $w_2(C)=0$ then there is an additional $\Z/2\Z$-normal fiber subgroup in $\pi_n(X)$. Since $\phi(S^n)$ represents the trivial class in $H_n(X;\Z)$, the exact sequence of Lemma \ref{lem:pi_n} -- and the fact that the map $\eta$ occuring in that exact sequence  is represented by a map into the $(n-1)$-skeleton --  implies that $\phi$ may be homotoped to land in the $(n-1)$-skeleton of $X$, which is the fiber linking $(n-1)$-sphere $\nu$ of $C$ by Lemma \ref{Lem:Thom_cell}.  

Since $\BB$ is an $(n+1)$-manifold with boundary, it is homotopy equivalent to an $n$-dimensional cell complex.  Suppose first that $\BB$ is simply-connected. Since $\BB$ has dimension $\geq 5$, we can find a Morse function $f\colon \BB \to \R$ with the following two properties:
\begin{itemize}
\item $f$ achieves its maximum along the boundary $\partial \BB \cong S^n$, and has outward-pointing gradient vector field along the boundary;
\item $f$ has no index $1$ or index $n$ critical points.
\end{itemize}
In this case, $\BB$ admits a cell structure with cells of dimension $\leq n-1$, and one can homotope the evaluation map $\ev\colon \BB \to X$ into the $(n-1)$-skeleton $\nu$.  The preimage of a regular value in $\BB$ gives a framed cobordism which by the Pontryagin-Thom construction shows that the map is homotopically trivial exactly as in Section \ref{ssec:pfthmyasha}.  

If $\BB$ is not simply connected, we may modify it by surgery, as in the proof of Corollary \ref{Cor:yasha} in Section \ref{Subsec:Pf_yasha}. Pick a collection of pairwise disjoint embedded loops $\gamma_i \subset \BB$ representing a generating set for (the finitely generated group) $\pi_1(\BB)$. We may assume that the evaluation map $\ev$ is an embedding in a neighbourhood of each $\gamma_i$. Since $\BB$ is spin, it is in particular orientable, so each of the loops $\gamma_i$ has trivial normal bundle, and an open neighborhood $U(\gamma_i) \cong S^1 \times D^n$. We claim that we can do surgery, cutting out these neighborhoods and re-gluing $2$-handles $D^2 \times S^{n-1}$, in such a way that the following two conditions hold:
\begin{itemize}
\item[$(i)$] the spin-structure of $T\BB$ extends  across the $2$-handles, to give a spin manifold $\BB'$ with $\partial \BB' \cong \partial \BB = S^n$;
\item[$(ii)$]the evaluation map $\ev\colon \BB \to X$ extends across the cobordism between $\BB$ and $\BB'$ to define an evaluation map $\ev\colon \BB' \to X$.
\end{itemize}
For the first point, the only obstruction to extending a spin structure over a $2$-handle in $\BB'$ is that the surgery circle should carry the bounding framing (equivalently, be compatible with the bounding spin structure).  This can be ensured in the construction of $\BB$, from Remark \ref{Rem:spin_bound}.  Since $X$ is simply connected, the loops $\ev(\gamma_i) \subset X$ bound disks, which for dimension reasons we may assume,  after small perturbation, are embedded and pairwise disjoint.  Fixing such disks, and picking trivializations of their normal bundles to model the surgery, one obtains an extension of $\ev$ over the 2-handles. This returns us to the case of simply connected $\BB$ treated previously. The theorem follows.
\qed

\section{Whitney sphere links}

\subsection{Relating invariants\label{Sec:Relating}} Let $n\geq 3$. For a Whitney sphere link, we claim the Hopf linking number, defined as an element in $[S^n \times S^n, S^{2n-1}] \cong \pi_{2n}(S^{2n-1}) = \Z/2\Z$, and the  $\Z/2\Z$-summand of the homotopy class defined by one component in the complement of the other, always agree.  

Consider the Lagrangian submanifold $i\colon C \hookrightarrow\R^{2n}_{\st}$ obtained by surgery on a Whitney sphere $\iota\colon S^n \to \R^{2n}_{\st}$ (i.e., one with a single double point). Then  
\begin{equation}\label{eq:surgeredhmtpy}
\pi_n(\R^{2n}_{\st}\setminus C)= 
\begin{cases} 
\Z/{2}\Z\oplus \Z & \text{if $n$ is even,}\\ 
\Z/2\Z & \text{if $n$ is odd.} 
\end{cases}
\end{equation}

In both cases the $\Z/2\Z$-summand is carried by a fiber $(n-1)$-sphere $\nu$ in the normal bundle of $C$. Consider a map $\phi\colon S^n \to \R^{2n}_{\st}\setminus C$ which is disjoint from $\iota(S^n)$, and which represents zero in the $\Z$-summand above in the even-dimensional case.

\begin{Lemma} \label{lem:relating}
The map $\phi$ represents the non-trivial element in $\Z/2\Z \subset\pi_n(\R^{2n}_{\st}\setminus C)$ if and only if the Hopf linking number of the link $\iota \sqcup \phi\colon S^{n}\sqcup S^{n}\to \R^{2n}$ is non-trivial.
\end{Lemma} 

\begin{proof}
Via a Hamiltonian isotopy supported in a small neighborhood of the double point, we take a model of $\iota(S^n)$ which is very flat near the double point.  More precisely, consider the standard model 
\[ 
\{x+iy \in \bR^n \oplus i\bR^n, \, |x|^2+y^2=1\} 
\]
of the Whitney sphere, where the preimages of the double point are $(0, \pm 1)$, and rescale the factors $\bR^n$ and $i\bR^n$ by $\epsilon^{-1}$ respectively $\epsilon$.  
Via this flattened model, we find a small $(n-1)$-sphere $K \subset C$, with a neighborhood $N(K) \cong  K \times (-\delta,\delta) \subset C$ such that $\iota(S^n)$ is obtained by removing $N(K)$ and adding two small flat disks $D_1, D_2$ intersecting at one point along the boundary of $C$. Critically, both $N(K)$  and  $D_1 \cup D_2$ can be taken to lie arbitrarily close to the double point $p \in \iota(S^n)$.  It follows that under the Gauss maps
	\[
	\scrG_{\iota\sqcup\phi}: S^n \times S^n \to S^{2n-1}, \quad \scrG_{i \sqcup \phi}\colon C \times S^n \to S^{2n-1}
	\]
	we can find an open neighbourhood $U$ of points in $S^{2n-1}$ whose preimage is disjoint from $D_1 \cup D_2$ respectively  $N(K)$, since the latter subsets each have Gauss image in an arbitrarily small (depending on the flattening parameter $\varepsilon$) neighborhood of the $n$-dimensional subset $\scrG_{\iota\sqcup\phi}(\{p\} \times S^n)$.   
	
Suppose now $\phi$ defines the non-trivial element in $\Z/2\Z$, so $\phi$ is homotopic in $\R^{2n}_{\st}\setminus C$ to a map  $\psi\colon S^n \to \nu$ with non-zero Hopf invariant. The Gauss map
\[ 
\Gamma\colon C\times \nu \to S^{2n-1},
\]
which takes a pair of points to the unit vector of the oriented line segment connecting them,
has degree $\pm 1$, and, provided $\nu$ is sufficiently small, for a generic point $u\in S^{2n-1}$, $\Gamma^{-1}(u) = (q,n) \in C\times \nu$ will comprise a single point.  We pick $u$ in the open subset $U$ above. For the Gauss map $\scrG_{i \sqcup \psi}$, the preimage 
	\[
	\scrG_{i \sqcup \psi}^{-1}(u) \, = \, (q, \psi^{-1}(n))
	\]
	of the point $u$ is exactly the framed knot fiber of $\psi$ over the point $n$, equipped with the framing it carries from $\psi$.  Furthermore, since $u \in U$, the discussion of the first paragraph shows that this framed knot lies in a ball in $S^n \times S^n$.  This shows that the Hopf linking invariant of $\iota \sqcup \psi \simeq \iota \sqcup \phi$ is also non-trivial. The choice of $u \in U$ ensures that the preimages of $\iota\sqcup \phi$ and $i\sqcup \phi$ can be identified, so the lemma follows.
\end{proof}

\subsection{Proof of Theorem \ref{Thm:whitney_link}}
Suppose we have a Whitney sphere link $\iota\colon S^n_a \sqcup S^n_b \to \R^{2n}$, $n>4$, with each component having a single  double point of Maslov index $n$. Let $C$ denote the Maslov index $n$ surgery of $S^n_b$. If $n$ is even then $C\approx S^{1}\times S^{n-1}$ and if $n$ is odd then $C$ is diffeomorphic to the non-orientable $S^{n-1}$-bundle over $S^{1}$.

In either case, $C\subset\R^{2n}_{\st}$ is a monotone Lagrangian submanifold of minimal Maslov index $n$. Equipping the sphere bundle with the standard metric on the sphere factor (which is invariant under the monodromy map, which we can take to be an isometry in both cases) we find that the minimal index of a non-constant contractible geodesic equals $n-2\ge 3$. In particular, we find by Theorem \ref{Thm:embLag} that the map $\iota\colon S^n_a \to \R^{2n}_{\st}\setminus C$ is homotopic to a constant map in the complement of $C$. In particular, in the even-dimensional case, $\iota(S^n_a)$ represents zero in the $\Z$-summand of $\pi_n(\R^{2n}_{\st} \setminus C)$.  This satisfies the hypotheses of Lemma \ref{lem:relating}, which implies the result. 


\subsection{Whitney links in four-dimensional space}  Let $\iota_1 \sqcup \iota_2\colon S^2 \sqcup S^2 \to \R^4$ be a Whitney link.  Kirk \cite[Section 6]{Kirk} proves that in this dimension, the Hopf linking number can be understood as follows. Let $\gamma$ be a loop on the first component which generates $\pi_1(\iota_1(S^2))$. Then 
\begin{equation} \label{eqn:Kirk}
\mathrm{Hopf}(\iota_1,\iota_2) = 0 \, \Leftrightarrow \, [\gamma] = 0 \in H_1(\R^{4} \backslash \iota_2(S^2);\Z)
\end{equation}
(this is the content of his formula for $\sigma_+$ on p.~685 of \emph{op. cit.} when there is a unique double point on each component). Using this characterization, we show:

\begin{Proposition} \label{Prop:2dim}
A Lagrangian Whitney sphere link in $\R^{4}$ has trivial Hopf invariant.
\end{Proposition}

\begin{proof}
Note that the Maslov index of a Lagrangian Whitney sphere is necessarily $2$ in this dimension.  We consider as usual holomorphic disks with one positive puncture in the generating class $\beta \in \pi_2(\R^4, \iota_1(S^2))$ for the first component.  From the usual Gromov-Floer displacing argument, it follows that the moduli space of such discs is non-empty for every taming $J$.  From \eqref{eqn:Kirk}, to prove that the Hopf linking number vanishes, it suffices to prove that such a $J$-holomorphic disk $D$ on $\iota_1(S^2)$ has trivial algebraic intersection number with the second component $\iota_2(S^2)$.  Performing a Lagrange surgery on $\iota_2(S^2)$ whose trace is disjoint from $\iota_1(S^2)$, we can replace the second Whitney sphere by an embedded Lagrangian torus $L$, and it suffices to prove that the holomorphic disc $D \in H_2(\bC^2, \iota_1(S^2))$ has trivial homological intersection number with $L \subset \R^{4}\backslash \iota_1(S^2)$.
 
We now employ a simple trick due to Welschinger \cite{Welschinger}.  Stretch the neck along the boundary of a small disk cotangent bundle $U^*L$ as usual.  The holomorphic disk $D$ will break into various components, say $E_1,\ldots, E_l$ inside $T^*L$, and $E_1',\ldots,E_j'$ in $\R^4\backslash L$, where $E_1'$ has boundary on $\iota_1(S^2)$ and all the other components are punctured spheres. The projection $p_*\colon H_1(U^*L;\Z) \to H_1(L;\Z)$ is an isomorphism when restricted to the subspace $\mathcal{R} \subset H_1(U^*L)$ spanned by Reeb orbits for the Reeb flow associated to the flat metric on $T^2$, see \cite{Welschinger}. For each component $E_i \subset T^*L$, the total boundary $R_i$ of $E_i$ is nullhomologous (via $E_i$ itself) in $T^*L$, hence the projection $p_*(R_i) \in H_1(L;\Z)$ of $R_i$ to the zero-section $L$ is nullhomologous. The compactness theorem in SFT shows that $R_i$ is a union of Reeb orbits and hence $R_i \in \mathcal{R}$, so as $p_*\colon \mathcal{R} \to H_1(L;\Z)$ is an isomorphism,  we deduce $R_i = 0\in H_1(U^*L;\Z)$.  We may therefore choose 2-chains $F_i \subset U^*L$ inside the spherical cotangent bundle with boundary $R_i$, for each $i$. Adding and subtracting these two-chains to the union of building components obtained from stretching $D$, we see that $D$ is homologous in $H_2(\R^{4}, \iota_1(S^2))$ to the sum 
\[
D \sim \, \left(\bigcup E_i \cup F_i\right) \cup \left(\bigcup E_j' \cup (-F_i)\right)
\]
where the final minus sign denotes orientation reversal.  The first term above is a closed 2-cycle in $T^*L$, which is therefore homologous to  some multiple of the zero-section $[L]$ itself, and has trivial self-intersection with $[L]$ (as $L$ has trivial Euler characteristic). The second term is a two-chain which is wholly contained in $\R^{4}\backslash L$, so obviously has trivial intersection with $[L]$. Therefore $D \cdot [L] = 0$, as required.
\end{proof}

\section{Linked examples}  \label{Sec:Linked}
Examples illustrate the necessity of the various hypotheses in Theorems \ref{Thm:embLag} and consequently in Theorem \ref{Thm:whitney_link}. Recall that $S_k \subset \R^n$ denotes a point, if $k=0$, or a sphere $S^k$ if $k>0$.

\begin{Lemma} For any non-trivial element  $\eta \in \pi_n(S^{n-k-1})$, there is a totally real embedding $\bR^n \hookrightarrow T^*(\bR^n\setminus S_k)$, co-inciding with a fixed cotangent fiber $T_x^*\R^{n}$ near infinity, for which the projection to the zero-section defines $\eta$. 
\end{Lemma}

\begin{proof}
This follows from the $h$-principle for $\epsilon$-Lagrangian embeddings (more precisely, for the ``extension" form of that $h$-principle, since we work relative to fixed data near infinity), cf. \cite{Gromov:PDR}.
\end{proof}

More strikingly, without constraints on the Maslov index of the double point, non-trivial Lagrangian Whitney sphere links do exist.  

\begin{Lemma} \label{Lem:whitney_link}
	If $n \geq 3$ is odd, there is a Lagrangian Whitney sphere link with non-vanishing Hopf linking number.
\end{Lemma}

\begin{proof}
	We fix  the standard Whitney sphere $w(S^n) \subset \R^{2n}_{\st}$, and construct a second immersed Lagrangian sphere which links it non-trivially. Let $L \subset \R^{2n}_{\st}$ be  the Lagrangian submanifold obtained by surgery on $w(S^n)$. Fix a point $p$ of $\R^{2n}_{\st}\setminus L$, and remove $L \sqcup \{p\}$ from $\R^{2n}_{\st}$ to obtain a symplectic manifold with disconnected negative end. Near the negative end asymptotic to $\{p\}$ we start with the standard Lagrangian disk bounded by the Legendrian unknot, and build a cobordism to a loose Legendrian sphere by pushing the two sheets of the unknot front through each other. The cobordism is immersed with precisely one double point, compare to \cite{YETI}. The positive end of this disk cobordism is now a loose Legendrian knot, which has a cap in each of the two possible homotopy classes in the complement of $L$, by \cite{EM}. The union of the cap and the immersed disk is our desired one double point immersion, which links $L$ and hence $w(S^n)$ non-trivially. The non-vanishing of the Hopf linking invariant then follows from Lemma \ref{lem:relating}. \end{proof}

The surgery on an immersed sphere with one low-index double point has vanishing Maslov class when $n=3$, and has minimal Maslov number $2-2k$ when $n=2k+1\geq 5$.  In particular, the examples constructed in the proof of Lemma \ref{Lem:whitney_link} show the necessity of the monotonicity and minimal Maslov index hypotheses in Theorem \ref{Thm:embLag}.

In $\R^{4k}$, the argument of Lemma \ref{Lem:whitney_link}, based on the results of \cite{YETI}, enables one to construct an immersed Lagrangian sphere with $3$ double points which non-trivially links the standard Whitney sphere.

\appendix

\section{Computation of a homotopy group}
In this appendix we compute the 
the $n^{\rm th}$ homotopy group of the complement of a Lagrangian submanifold of $\R^{2n}$. 

\subsection{The complement of a closed Lagrangian submanifold in $\R^{2n}$}\label{sec:algtop}

The stable homotopy type of the complement of a submanifold of Euclidean space is determined by Spanier-Whitehead duality.  We will give a concrete description of the homotopy group of relevance to us, in the special case of a Lagrangian submanifold (the key point is that the normal bundle agrees with the tangent bundle). 
Let $C\subset \R^{2n}_{\st}$, $n\ge 4$ be a Lagrangian submanifold and set $X=\R^{2n}\setminus C$. Fix a Morse function $f\colon C \to \R$ with exactly one minimum, and consider the corresponding CW-decomposition of $C$:
\[ 
C=C^{(n)}\supset C^{(n-1)}\supset\dots \supset C^{(0)} = \{\mathrm{pt}\},
\]
where $C^{(j)}$ is the union of the closures of the stable manifolds of all critical points of index $\le j$, and $\mathrm{pt}$ denotes a point. For dimensional reasons, the tangent bundle $TC$ has a non-zero section over the $2$-skeleton $C^{(2)}$. Fix such a section $v$, and consider the Thom space $Y_2$ of the spherical tangent bundle $SC$ over $C^{(2)}$, obtained from the restriction $SC|_{C^{(2)}}$ by collapsing the section $v$ to a point. 

\begin{Lemma}\label{lem:skeleton}
	$X$ is homotopy equivalent to a space obtained from $Y_2$ by attaching cells of dimension $\ge n+2$. 
\end{Lemma}

\begin{proof}
	Consider the inclusion $\R^{2n}\subset S^{2n}$ and $C\subset S^{2n}\subset S^{2n+1}$. Defining 
	\[
	\hat X=S^{2n}\setminus C \quad \textrm{and} \quad \hat X'=S^{2n+1}\setminus C
	\] 
	we note that there is a homotopy equivalence $\hat X' \simeq \Sigma \hat X$ between $\hat X'$ and the suspension of $\hat X$. Consider the height function $h$ on $S^{2n+1}$ for which the equator $S^{2n}$ is the zero-set. There is a  a tubular neighborhood $\nu\,C$ of $C$ on which the restriction of $h$ agrees with the height function on each fiber disk. Now introduce on each fiber $n$-disk of $\nu\,C$ a cancelling pair of critical points of indices $0$ and $1$. This gives a function on $S^{2n+1}$ with a pair of critical Bott manifolds $C$ and $C'$ of indices $0$ and $1$. We deform this Bott-Morse function so as to pull the local minimum submanifold below the minimum of $h$, further Morsifying the index one critical submanifold $C'$ via $f$. Then the super-level set of a regular value slightly larger than the new global minimum is homotopy equivalent to the complement of $C$; in this description it is clear that it is obtained by attaching one cell of dimension $n+k$, starting with the maximum $0$-cell, for each critical point of $-f$ of index $k$. 
	
	We conclude that $\hat X'$ has a cell structure with $(n+2)$-skeleton given by the standard cell-structure of the Thom space $Y'_2$ of the restriction of the normal bundle of $C$ in $S^{2n+1}$ to $C^{(2)}$, with cells of dimension $\ge n+3$ attached. Any non-zero section of the normal bundle of $C$ in $S^{2n}\subset S^{2n+1}$ is homotopic to the vertical normal vector field in $S^{2n+1}$, which implies that $Y_2'$ is the suspension of $Y_2$.
\end{proof}

We next describe the attaching maps in the $(n+1)$-skeleton of $Y_2$. To this end we consider the chain complex generated in degree $j$ by the $j$-cells in $C^{(j)}$, $j=0,1,2$, with the differential $d^{\mathrm{tw}}$ in cellular homology twisted by the orientation bundle $\det(TC)|_{C^{(2)}}$. We view $\det(TC)$ as a rank one local system with $\Z$-coefficients.

\begin{Lemma} \label{Lem:Thom_cell}
	$Y_2$ has a cell structure with one $0$-cell corresponding to the point at infinity in the Thom space and one $((n-1)+j)$-cell for each $j$-cell of $C^{(2)}$, $j=0,1,2$, so that:
	
	\begin{enumerate}
		\item The $(n-1)$-skeleton $Y_0$ is an $(n-1)$-sphere.
		\item The $n$-skeleton $Y_1$ is obtained from $Y_0$ by attaching one $n$-cell $Y_1^{\tau}$  for each $1$-cell $\tau$ in $C^{(1)}$. The degree of the attaching map of $Y_1^\tau$ equals the coefficient of $\epsilon_0$ in $d^{\mathrm{tw}}\tau$, where $\epsilon_0$ is the $0$-cell. 
		\item The $(n+1)$-skeleton $Y_2$ is obtained from $Y_1$ by attaching one $(n+1)$-cell $Y_2^{\sigma}$ for each $2$-cell $\sigma$ in $C^{(2)}$. The degree of the attaching map of $Y^2_{\sigma}$ to $Y^1_{\tau}$ equals the coefficient of $\tau$ in $d^{\mathrm{tw}}\sigma$.
	\end{enumerate}
	\end{Lemma}

\begin{proof} 
	Picking a non-zero section $v$ of the tangent bundle $TC$ over the 2-skeleton $C^{(2)}$ we can identify $Y_{2}$ with the Thom space of its orthogonal complement $v^{\perp}$. The $(n-1)$-skeleton of $Y_2$ is then the restriction of this Thom space to the point in $C^{(0)}\subset C^{(2)}$ and is clearly an $(n-1)$-sphere. 
	
	We next compute the degrees of the attaching maps of the higher cells. An $(n-1+j)$-cell $\hat\sigma^j$ is given by $\sigma^{j}\times D^{n-1}$ where $\sigma^{j}\approx D^{j}$ is a $j$-cell in $C$, where $D^{n-1}$ is the fiber of $v^{\perp}$ at the center of $\sigma^{j}$, and where the product structure is obtained by trivializing the pull-back of the bundle $v^{\perp}$ over the cell by moving the central fiber along rays in the disk. The attaching map $i$ of the cell $\hat\sigma^{j}$ takes $\sigma^{j}\times\partial D^{n-1}$ to the base point in the Thom space,  in other words to the $0$-cell, and takes $\partial\sigma^{j}\times D^{n-1}$ to the bundle over the $(j-1)$-skeleton by the natural map from $i^*v^{\perp}$ to $v^{\perp}|_{\mathrm{im}(i)}$.  The preimage in $\sigma^j \times D^{n-1}$ of the point given by the origin in $D^{n-1}$ at the central point in a cell $\sigma^{j-1}$ consists of the origin in $D^{n-1}$ at each of the preimages of the central point under the attaching map of $\sigma^{j}$ to $\sigma^{j-1}$. Furthermore, the local degree at such a preimage is exactly the local degree of the attaching map $\sigma^{j}\to\sigma^{j-1}$ multiplied by the sign of the parallel transport from the fiber at the central point of the disk. Since $\det(TC)|_{C^{(2)}} \cong \det v^{\perp}|_{C^{(2)}}$, this is exactly the twisted differential in the cellular complex for $C$.
	\end{proof}
	
	\begin{Corollary}\label{Cor:HnX}
	There is an isomorphism $H_n(X;\Z) \cong H_1(C;\Z\otimes \det(TC))$.
	\end{Corollary}


As a first step in computing $\pi_{n}(X)$ we consider the Hurewicz homomorphism.

\begin{Lemma} \label{Lem:hurewicz} 
	The Hurewicz homomorphism $h\colon \pi_n(X) \to H_n(X;\Z)$ is surjective.
\end{Lemma}

\begin{proof}
	Inclusion gives isomorphisms $\pi_n(Y_2)\cong \pi_n(X)$ and $H_n(Y_2)\cong H_{n}(X)$, and hence it suffices to prove the result for $Y_2$. It follows from the cell structure of Lemma \ref{Lem:Thom_cell} that any $n$-cycle $\Gamma$ in $Y_2$ corresponds to a $1$-cycle for the orientation-twisted differential on the 1-skeleton $C^{(1)}$. Any such cycle is carried by an orientation preserving loop $\gamma$ in the 1-skeleton (i.e. one such that $TC|_{\gamma}$ is trivial). Parameterizing $\gamma$ by an interval $[0,1]$, we define a map $f_{\gamma}$ from the sphere into $Y_1$,
	\begin{equation} \label{eq:f_explicit}
	f_{\gamma}\colon S^{n} \ \approx \ ([0,1]\times D^{n-1})/\partial ([0,1]\times D^{n-1}) \  \longrightarrow \ Y_1,
	\end{equation}
	exactly as when we defined the attaching maps corresponding to $1$-cells. More precisely, $f_{\gamma}$ takes $[0,1]\times \partial D^{n-1}$ to the base point and maps $[0,1]\times \mathrm{int}(D^{n-1})$ to $Y_1$ using parallel translation in the pull-back of the bundle $v^{\perp}$ over $[0,1]$ from the central point $1/2 \in [0,1]$. 
	The image of the map $f_{\gamma}$ sweeps the cellular cycle $\Gamma$  with multiplicity one, and so its image under the Hurewicz homomorphism is exactly $\Gamma$. The lemma follows. 
\end{proof}

Let $\eta \in \pi_n(Y_0) = \Z/2\Z$ denote a generator. Via inclusion $Y_0 \subset Y_2$, this defines a class (not necessarily non-trivial) in $\pi_n(Y_2) \cong \pi_n(X)$.

\begin{Lemma} \label{Lem:kernel}
The kernel of the Hurewicz homomorphism $h$ is generated by $\eta$.
\end{Lemma}

\begin{proof}
The homotopy exact sequence for $(Y_2,Y_1)$ includes
	\begin{equation} \label{eqn:les_12}
\cdots \to \pi_{n+1}(Y_2,Y_1) \to \pi_n(Y_1) \to \pi_n(Y_2) \to \{0\}.
\end{equation}
Suppose that $f: S^n \to Y_2$ is in the kernel of $h$. We may assume from \eqref{eqn:les_12} that $f$ has image in $Y_1$, and we wish to prove that we can homotope $f$ into the $(n-1)$-skeleton $Y_0$, since $\eta$ generates $\pi_n(Y_0)$.   To do this, it suffices to show that we can homotope $f$ to a map $f'$ so that, for each $n$-cell in $Y_1$, the degree of the composition of $f'$ with projection to that $n$-cell vanishes.

 Let $[f(S^n)]$ denote the cellular $n$-cycle defined by $f$.  There is a cellular $(n+1)$-chain $Y_{2}(\sigma)$ with $\partial Y_2(\sigma)=[f(S^n)]$. Using the isomorphism with the twisted homology complex, this cellular $(n+1)$-chain corresponds to a cellular $2$-chain $\sigma$ for $C^{(2)}$. Suppose $d^{\mathrm{tw}}\sigma=\gamma$. 
	
	Let $\tau$ be a 2-cell which appears with coefficient $-n_\tau$ in $\sigma$. Let $\partial \tau$ denote the loop parameterized by the attaching map of $\tau$. Then $f$ is homotopic to $f \ast f_{\partial\tau}$ where $\ast$ denotes composition in $\pi_n(Y_2)$ and $f_{\partial \tau}$ is as in \eqref{eq:f_explicit}. The homotopy pulls the part of the map corresponding to $f_{\partial\tau}$ across the $2$-cell $\tau$. The $n$-cycle $[f \ast f_{\partial \tau}(S^n)]$ has associated $1$-chain $\gamma + \partial \tau$ in the twisted cellular complex for $C$, and there is a chain $\sigma'$ with $d^{\mathrm{tw}}\sigma'=\gamma+\partial\tau$. The coefficient of $\tau$ in $\sigma'$ is $-(n_{\tau}-1)$. Continuing  inductively, we homotope the map $f$ to a map $f'$ whose associated cellular $n$-cycle $[f'(S^n)] = \partial Y_2(0)$, i.e. where the bounding cellular $(n+1)$-chain is the zero-chain.  This precisely means that the map obtained by composing $f'$ with projection to any $n$-cell in $Y_1$ has degree $0$, as required. \end{proof}

Combining the previous two Lemmas, to compute $\pi_n(X)$, it only remains to understand when $\eta$ vanishes in $\pi_n(X) = \pi_n(Y_2)$.

For $i=1,2$ let $w_i(TC)$ denote the $i$-th Stiefel-Whitney class of $TC$, viewed as defining an element
\begin{equation} \label{eqn:sw_class_as_homomorphism}
w_i \in \Hom(H_i(C;\Z/2\Z);\Z/2\Z).
\end{equation}
Note that $C$ is orientable if $w_1=0$ and both orientable and spin if $w_1=0$ and $w_2=0$.   More concretely, $w_2=0$ if and only if the trivialization of $v^{\perp}$ over the orientation-preserving loops in the 1-skeleton of $C$ extends over all $2$-cells.  Thus, if both $w_1=0$ and $w_2=0$, we can trivialise $v^{\perp}$ over $C^{(2)}$, and $Y_2$ is the Thom space of a trivial bundle. In this case, the class $\eta$ obviously survives as a split summand in $\pi_n(Y_2)$, which shows that 
	\[
	\pi_n(Y_2) = \Z/2\Z \langle \eta \rangle \oplus H_n(X;\Z).
	\]
	To prove the general statement, recall that $\pi_n(Y_2)$ is a quotient of $\pi_n(Y_1)$. 
	
	\begin{Lemma}
The homotopy group $\pi_n(Y_1) = H_1(C^{(1)};\Z\otimes \det(TC))\oplus \Z/2\Z\langle\eta\rangle.$
\end{Lemma}

\begin{proof}
Consider the $1$-skeleton $C^{(1)}$; the corresponding $n$-skeleton $Y_1$ is constructed as follows. Start with an $(n-1)$-sphere $\Sigma$,  for each orientation preserving $1$-handle in $C^{(1)}$ attach an $n$-disk to $\Sigma$ with an attaching map of degree $0$, and for each orientation reversing handle attach an $n$-disk to $\Sigma$ with attaching map of degree $2$. Homotoping the attaching maps of degree $0$ to constant maps, we see that $Y_1$ is homotopy equivalent to a wedge $Z \vee \bigvee_j S^n_j$, where there is one $S^n$-factor for each orientation preserving 1-handle, and where the space $Z$ is obtained from $\Sigma$ by attaching $n$-disks by maps of degree $2$. 

Note that $\pi_n(Y_1,Y_0) \cong \Z^k$, where $k$ is the number of $1$-cells in $C^{(1)}$, and the map $\chi: \pi_n(Y_1,Y_0) \to \pi_{n-1}(Y_0) = \Z$ sends the generator associated to a 1-cell $\gamma_j$ to the attaching degree $d^{\mathrm{tw}}({\gamma_j}) = d_j$.  Thus, $\chi$ is naturally identified with the boundary map $C^{\textrm{tw}}_1 \to C^{\textrm{tw}}_{0}$ in the orientation-twisted cellular complex for $C^{(1)}$. 
The homotopy exact sequence for the pair $(Y_1,Y_0)$ reads
\begin{equation} \label{eqn:les_01}
 \oplus_j \Z/2\Z \stackrel{\alpha}{\longrightarrow} \pi_n(Y_0) = \Z/2\Z  \to \pi_n(Y_1) \to \pi_n(Y_1,Y_0) \to \Z \to \Z/\{d_j\}\Z \to \{0\},
\end{equation}
where we have used
\[ 
 \pi_{n+1}(Y_1,Y_0) = \pi_{n+1}(\bigvee_j S^n) = \oplus_j \Z/2,
\]
which follows from Hilton's theorem (since $n\geq 3$, all non-trivial Whitehead products have degree $\geq 2n-1 > n+1$). We claim the map $\alpha $ in \eqref{eqn:les_01} vanishes. 
It suffices to treat the case when there is a single $n$-cell attached by a map $g$ of degree $2$, and $Y_1 = Z$ is the Moore space $M(\Z/2,n-1)$. There is a cofiber sequence $S^{n-1} \to C_{g} \to Z$, where $C_{g} \simeq S^{n-1}$ is the mapping cylinder of $g$;  this yields an exact sequence, for $i \leq 2n-2$,
\[
\cdots \to\pi_i(S^{n-1}) \stackrel{2}{\longrightarrow} \pi_i(S^{n-1}) \to \pi_i(Z) \to \pi_{i-1}(S^{n-1}) \stackrel{2}{\longrightarrow} \pi_{i-1}(S^{n-1})\to \cdots
\] which shows that $\pi_{n}(Z) = \Z/2$ as required.  Since $\ker(\chi)$ is free abelian, the sequence
\begin{equation} \label{eqn:n-skeleton}
0 \to \Z/2 \to \pi_n(Y_1) \to \ker(\chi) \to 0
\end{equation}
necessarily splits, which completes the proof.
	\end{proof}

\begin{Lemma}\label{lem:pi_n}
	The $n$-th homotopy group $\pi_n(X) \cong \pi_n(\R^{2n} \setminus C)$ is determined as follows:
	\begin{itemize}
		\item If  $w_2 = 0$, there is an exact sequence 
		\[
		0 \to \Z/2\langle \eta \rangle  \to \pi_n(X) \stackrel{h}{\longrightarrow} H_n(X;\Z) \to 0,
		\]
		which splits canonically and hence 
		\[
		\pi_n(X) \approx \bZ/2\langle \eta\rangle \oplus H_1(C;\Z\otimes\det(TC)).
		\]
		\item If $w_2 \neq 0$, then  $h$ is an isomorphism and hence
		\[  
		\pi_n(X) \approx H_1(C;\Z\otimes\det(TC)).
		\]
	\end{itemize}
\end{Lemma}

\begin{proof}

 Given any cell complex $A$ with only $j$-cells for $j\in\{0, n-1,n,n+1\}$, if $A'$ is obtained by attaching an $(n+1)$-cell to $A$ along $f\colon S^{n}\to A$, then
	\[ 
	\pi_n(A')=\pi_n(A)/\langle[f]\rangle.
	\]
	Consider attaching a $2$-disk to $C^{(1)}$. Assume first that the trivialization of $v^{\perp}$ over orientable loops in the $1$-skeleton of $C$ extends over the attached 2-disk. By our description of the attaching maps for $Y_2$ via transporting trivializations from the central point in cells radially outwards, the attaching map of the corresponding $(n+1)$-cell is the sum of the $n$-spheres corresponding to the $1$-cells in the boundary of the given 2-disk. This vanishes in the $\Z/2\Z$-summand of $\pi_n(Y_1)$.  If $w_2=0$, the result follows, with the canonical splitting of $\pi_n(Y_2)$ inherited from that for $\pi_n(Y_1)$.

	To conclude, it suffices to show that if $w_2\neq0$ then $\eta$ arises as the homotopy class of the attaching map of some $(n+1)$-cell when constructing $Y_2$ from $Y_1$. 
We know there is a $2$-cell $\sigma$ over which the trivialization of $v^{\perp}$ does not extend, and inductively we may suppose this is the first attached handle over which the trivialization does not extend.   Let $\psi$ denote the attaching map of the $(n+1)$-cell in $Y_2$ corresponding to $\sigma$. The boundary loop $\gamma$ of the $2$-handle $\sigma$ must be expressible in terms of boundaries of $2$-handles that have already been attached (in particular, $\gamma$ is bounding). Consider now the map $f_{\gamma}$ from Lemma \ref{Lem:hurewicz}.  Since the framing of $v^{\perp}$ does not extend over $\sigma$, we find that $\psi$   is homotopic to the map $f_{\gamma}'$ which is defined in the same manner as $f_{\gamma}$ but using a map $[0,1]\times D^{n-1} \to Y_1$ which differs from that of $f_{\gamma}$ by a twist corresponding to the generator of $\pi_1(O(n-1))$. 
	
	We now homotope $f_{\gamma}'$ across the two cell $\sigma$ to a map  into $\Sigma$. Identify $\Sigma$ with the one point compactification of the fiber in $v^{\perp}$. As before, we homotope the map by pulling the loop $\gamma$ over the disk and parallel translating the fiber. The preimage of the origin in $\Sigma$  is the boundary loop $\gamma$ with its initial framing. In the $n$-sphere boundary of the attached $(n+1)$-handle this is the non-trivially framed loop, so we conclude that $\eta$ is in the image of the attaching map. Hence the presence of a $2$-disk in $Y_2$ over which the trivialization of $v^{\perp}$ does not extend implies that $\eta = 0 \in \pi_n(Y_2)$.
	\end{proof}

\subsection{The complement of a Lagrangian submanifold with Legendrian boundary in $\R^{2n}_{\st}$}\label{sec:withboundary}
In this section we carry over the study for closed Lagrangians in Section \ref{sec:algtop} to the case of Lagrangians with Legendrian boundary. As in Section \ref{sec:algtop} let $C$ denote the Lagrangian and let $\Gamma$ denote its Legendrian boundary. We show here that Lemma \ref{lem:pi_n} holds unchanged in this situation. In other words $\pi_n(\R^{2n}_{\st}\setminus C)$ is isomorphic to
\[ 
H_{1}(C,\Z\otimes\det(TC))\oplus\Z/2\Z\langle\eta\rangle,
\]
if $w_2(C)=0$, and 
\[ 
H_{1}(C,\Z\otimes\det(TC)),
\]
if $w_2(C)\ne 0$.

To see this we consider $C$ as a clean submanifold of $D^{2n}$ with boundary $\Gamma\subset S^{2n-1}$.
The complement of $C\subset D^{2n}\subset D^{2n+1}$ is again the suspension of its complement in $D^{2n}$. 

We will consider Morse functions on manifolds with boundary which are extensions of Morse functions on the boundary, and for which the Morse flow in the boundary agrees with the Morse flow on the whole manifold. As in the closed case, we consider a deformation of the Morse function on $D^{2n+1}$ with an index $2n+1$ and an index $1$ critical point on the boundary sphere and a minimum at the center. We deform this by introducing a minimum and an index one critical point in each normal fiber of $C$, where the normal fiber over $\Gamma$ is the normal fiber in the boundary. Exactly as in the closed case, we find that the $(n+2)$-skeleton of the complement is given by the Thom space of the normal bundle over the $2$-skeleton on $C$  defined by the stable manifolds of a Morse function on $C$ which has its unique minimum on the boundary. After observing that the Morse complex associated to such a function computes the homology of $C$, the result follows from a verbatim repetition of the argument in Section \ref{sec:algtop}.

\bibliographystyle{alpha}

\begin{thebibliography}{10}


\bibitem{Abouzaid:htpy}
M.~Abouzaid.
\emph{Nearby Lagrangians with vanishing Maslov class are homotopy equivalent.}
Invent. Math. {\bf 189} (2012) 251--313.


\bibitem{AbSm}
M.~Abouzaid and I.~Smith.
\emph{Exact Lagrangians in plumbings.}
Geom. Funct. Anal.  {\bf 22} (2012) 785--831.

\bibitem{Abouzaid-Seidel}
M.~Abouzaid and P.~Seidel.
\emph{An open string analogue of Viterbo functoriality.}
Geom. Topol. {\bf 14} (2010) 373--440.

\bibitem{AbKr}
M.~Abouzaid and T.~Kragh.
\emph{Simple homotopy equivalence of nearby Lagrangians}.
Preprint, available at arXiv:1603.05431.

\bibitem{AbKr-2}
M.~Abouzaid and T.~Kragh.
\emph{On the immersion classes of nearby Lagrangians.}
J. Topol. {\bf 9} (2016) 232--244.


\bibitem{audin-lafontaine}
M.~Audin and J.~Lafontaine (eds.), \emph{Holomorphic curves in
symplectic
  geometry}, Progress in Mathematics, vol. 117, Birkh{\"a}user, 1994.

\bibitem{BEE}
F.~Bourgeois, T.~Ekholm and Y.~Eliashberg.
\emph{Effect of Legendrian surgery},  
Geom. Topol. {\bf 16} (2012) 301--389.

\bibitem{BEHWZ}
F.~Bourgeois, Y.~Eliashberg, H.~Hofer, K.~Wysocki and E.~Zehnder.
\emph{Compactness results in symplectic field theory.}
Geom. Topol. {\bf 7} (2003) 799Ð888.


\bibitem{Cieliebak:chord}
K.~Cieliebak.
\emph{Handle attaching in symplectic homology and the chord conjecture.}
 J. Eur. Math. Soc. (JEMS) {\bf 4} (2002) 115--142.

\bibitem{CEL}
K.~Cieliebak, T.~Ekholm and J.~Latschev. 
\emph{Compactness for holomorphic curves with switching Lagrangian boundary conditions.} 
J. Symplectic Geom. {\bf 8} (2010) 267--298.

\bibitem{CE}
K.~Cieliebak, Y.~Eliashberg,
\emph{From {S}tein to {W}einstein and back},
American Mathematical Society Colloquium Publications {\bf 59},
American Mathematical Society, Providence, RI (2012).


\bibitem{CNS}
C.~Cornwell, L.~Ng, S.~Sivek, 
\emph{Obstructions to Lagrangian concordance}, 
Algebr. Geom. Topol. {\bf 16} (2016) 797--824. 


\bibitem{DRE}
G.~Dmitroglou Rizell and J.~Evans.
\emph{Unlinking and unknottedness of monotone Lagrangian submanifolds.}
Geom. Topol. {\bf 18} (2014) 997--1034

\bibitem{E}
T.\ Ekholm, 
\textit{Non-loose Legendrian spheres with trivial Contact Homology DGA}, 
Preprint, arXiv:1502.04526, to appear in J.\ Topol.

\bibitem{Ethesis}
T.\ Ekholm, 
\textit{Immersions in the metastable range and spin structures on surfaces.}, 
Math. Scand. {\bf 83} (1998) 5--41.

\bibitem{YETI}
T.~Ekholm, Y.~Eliashberg, E.~Murphy and I.~Smith.
\emph{Constructing exact Lagrangian immersions with few double points.}
Geom. Funct. Anal.  {\bf 23} (2013) 1772--1803.

\bibitem{EES}
T.~Ekholm, J.~Etnyre, M.~Sullivan.
\emph{Non-isotopic Legendrian submanifolds in $\R^{2n+1}$.}
J. Differential Geom. {\bf 71} (2005) 85--128. 

\bibitem{EkholmKraghSmith}
T.~Ekholm, T.~Kragh, I.~Smith. 
\emph{Lagrangian exotic spheres.} 
J. Topol. Anal. {\bf 8} (2016) 375--397.

\bibitem{EL}
T.~Ekholm, Y.~Lekili,
\emph{Duality between Lagrangian and Legendrian invariants.}
Preprint, arXiv:1701.01284.


\bibitem{ENS}
T.~Ekholm, L.~Ng and V.~Shende.
\emph{A complete knot invariant from contact homology.}
Preprint, arXiv:1606.07050.

\bibitem{EkholmSmith}
T.~Ekholm and I.~Smith,
\emph{Exact Lagrangian immersions with a single double point},
J. Amer. Math. Soc. {\bf 29} (2016) 1--59.

\bibitem{EkholmSmith2}
T.~Ekholm and I.~Smith,
\emph{Exact Lagrangian immersions with one double point revisited},
Math. Ann. {\bf 358} (2014) 195--240.

\bibitem{EGH}
Y.~Eliashberg, A.~Givental and H.~Hofer.
\emph{Introduction to symplectic field theory.}
Geom. Funct. Anal. 2000 (Special Volume, II), 560--673.

\bibitem{EM}
Y.~Eliashberg and E.~Murphy.
\emph{Lagrangian caps.}
Geom. Funct. Anal. {\bf 23} (2013) 1483--1514.

\bibitem{Yasha}
Y.~Eliashberg.
\emph{Personal communication.} Institut Mittag-Leffler, September 2015.


\bibitem{FR}
R.~Fenn and D.~Rolfsen.
\emph{Spheres may link homotopically in 4-sphere.}
J. Lond. Math. Soc. {\bf 34} (1986) 177--184.

\bibitem{FSS}
K.~Fukaya, P.~Seidel and I.~Smith.
\emph{Exact Lagrangian submanifolds in simply-connected cotangent bundles.}
Invent. Math. \textbf{172} (2008) 1--27.

\bibitem{Gromov:PDR}
M.~Gromov.
\emph{Partial differential relations.}
Springer, 1986.

\bibitem{haefliger}
A.~Haefliger,
\emph{Plongements différentiables de vari\'et\'es dans vari\'et\'es.}  
Comment. Math. Helv. {\bf 36} (1961) 47--82. 

\bibitem{Kirk}
P.~Kirk.
\emph{Link homotopy with one codimension two component.}
Trans. Amer. Math. Soc. {\bf 319} (1990), 663--688. 

\bibitem{Kragh}
T.~Kragh, 
\emph{Parametrized ring-spectra and the nearby {L}agrangian conjecture.}
With an appendix by M.~Abouzaid,
Geom. Topol. \textbf{17} (2013) 639--731.

\bibitem{Massey-Rolfsen}
W.~Massey and D.~Rolfsen.
\emph{Homotopy classification of higher-dimensional links.}
Indiana Univ. Math. J. {\bf 34} (1985) 375--391.


\bibitem{Viterbo}
C.~Viterbo.
\emph{A new obstruction to embedding Lagrangian tori.}
Invent. math. {\bf 100} (1990), 301--320.

\bibitem{Welschinger}
J.-Y.~Welschinger.
\emph{Effective classes and Lagrangian tori in symplectic four-manifolds.} 
J. Symplectic Geom. {\bf 5} (2007),  9--18.

\end{thebibliography}

\end{document}